\documentclass[12pt, reqno]{amsart}

\pdfoutput=1

\usepackage[margin=3.14cm]{geometry}
\usepackage{mathtools, amssymb}
\usepackage{enumerate}
\usepackage{bm}
\usepackage{tikz-cd}
\usepackage[nocompress]{cite}
\usepackage{hyperref}
\usepackage{subcaption}
\usepackage{graphicx}
\usepackage{graphbox}
\usepackage[foot]{amsaddr}
\usepackage{microtype}

\theoremstyle{plain}
\newtheorem{theorem}{\bf Theorem}[section]
\newtheorem*{theorem*}{\bf Theorem}
\newtheorem{proposition}[theorem]{\bf Proposition}
\newtheorem{lemma}[theorem]{\bf Lemma}
\newtheorem{lemma-def}[theorem]{\bf Lemma--Definition}
\newtheorem{corollary}[theorem]{\bf Corollary}
\theoremstyle{definition}
\newtheorem{definition}[theorem]{\bf Definition}
\newtheorem*{definition*}{\bf Definition}

\newtheorem{problem}[theorem]{\bf Problem}
\newtheorem{conjecture}[theorem]{\bf Conjecture}
\newtheorem{example}[theorem]{\bf Example}

\newcommand		{\p}[1]	{\left(#1\right)}

\newcommand		{\abs}[1]{\left|#1\right|}
\newcommand		{\floor}[1]{\left\lfloor#1\right\rfloor}
\newcommand    	{\Z} {\mathbb Z}
\newcommand    	{\R} {\mathbb R}

\newcommand 	{\E} {\mathcal E}
\newcommand 	{\F} {\mathcal F}

\newcommand   {\homol} {{\mathrm H}}
\newcommand 	{\x} {{\bm{x}}}
\newcommand 	{\y} {{\bm{y}}}

\renewcommand 	{\v} {{\bm{v}}}

\newcommand 	{\ggamma} {{\bm{\gamma}}}
\newcommand 	{\ddelta} {{\bm{\delta}}}
\newcommand 	{\f} {{\bm{f}}}
\newcommand 	{\g} {{\bm{g}}}
\renewcommand 	{\d} {{\bm{d}}}
\newcommand 	{\0} {{\bm{0}}}
\newcommand 	{\1} {{\bm{1}}}
\newcommand 	{\flow} {\Phi}
\newcommand 	{\vertexspace} {{\R^V}}
\newcommand 	{\edgespace} {{\R^E}}
\newcommand 	{\equilibria} {{X_{G,f}}}
\newcommand 	{\argequilibria} [2] {X_{#1, #2}}
\newcommand 	{\cc} {\mathrm{cc}}
\newcommand 	{\trans} {{\mathsf{T}}}
\renewcommand 	{\ker} {\mathrm{Ker}}
\newcommand 	{\aut} {\mathrm{Aut}}

\newcommand    	{\anal} {\mathcal A}

\title[Graph Gradient Diffusion]
{From Combinatorics to Geometry: \\ The Dynamics of Graph Gradient Diffusion}
\author{Davide Sclosa}
\address{Vrije Universiteit Amsterdam, De Boelelaan 1111, 1081 HV Amsterdam, The Netherlands.
{email: \texttt{davide.sclosa@gmail.com}}. Orcid: 0000-0003-0806-2591.}
\date{\today}

\begin{document}

\begin{abstract}
We discuss a link between graph theory and geometry
that arises when considering graph dynamical systems with odd interactions.
The equilibrium set in such systems is not
a collection of isolated points, but rather a union of manifolds, which may
intersect creating singularities and may vary in dimension.
We prove that geometry and stability of such manifolds are governed by combinatorial properties
of the underlying graph. In particular, we derive an upper bound on the dimension of the
equilibrium set using graph homology and a lower bound using graph coverings.
Moreover, we show how graph automorphisms relate to geometric singularities and prove that
the decomposition of a graph into $2$-vertex-connected components
induces a decomposition of the equilibrium set that preserves three notions of stability.
\end{abstract}

\maketitle

%%%%%%%%%%%%%%%%%%%%%%%%%%%%%%%%%%%%%%%%%%%%%%%%%%%%%%%%%%%%%%%%%%%%%%%%%%%%%%%%%%%%%%%%%%%%%%%%
%%%%%%%%%%%%%%%%%%%%%%%%%%%%%%%%%%%%%%%%%%%%%%%%%%%%%%%%%%%%%%%%%%%%%%%%%%%%%%%%%%%%%%%%%%%%%%%%
%%%%%%%%%%%%%%%%%%%%%%%%%%%%%%%%%%%%%%%%%%%%%%%%%%%%%%%%%%%%%%%%%%%%%%%%%%%%%%%%%%%%%%%%%%%%%%%%

\section{Introduction}

Throughout this paper~$G$ is a finite graph and~$f: \R \to \R$ a non-constant
odd real-analytic function.
For each vertex~$i$ let~$x_i \in \R$ evolve according to the vector field
\begin{equation} \label{eq:main}
	\dot x_i = \sum_{j\in N(i)} f(x_j - x_i),
\end{equation}
where~$N(i)$ denotes the set of neighbors of the vertex~$i$ in the graph~$G$.
Let~$\equilibria$ denote the set of equilibria of~\eqref{eq:main} up to translational symmetry.
In this paper we analyze geometry and stability of the real-analytic set~$\equilibria$
via combinatorial properties of the graph~$G$.
As such, our results provide links between graph theory,
geometry, and dynamical systems.
For example, we obtain an upper bound on the dimension of~$\equilibria$
from the homology of the graph~$G$:

\begin{theorem} \label{thm:intro_snake}
Suppose that~$f$ is a polynomial.
Let~$\homol_1(G)$ be the first homology group\footnote{In this paper homology is always over the
reals~$\homol_1(G) = \homol_1(G,\R)$. By the universal coefficient
theorem~$\dim \homol_1(G,\R)$ is the same if~$\R$ is replaced by~$\Z$,
or by the common choice in graph theory~$\Z/2\Z$.} of the graph~$G$.
Let~$C_1,\ldots,C_p$ be cycle subgraphs of~$G$ such that for all~$i=1,\ldots,p-1$
the cycle~$C_i$ intersects~$C_{i+1}$ in exactly one edge
and there are no other
edge-intersections among~$C_1,\ldots,C_p$.
Then
\[
	\dim \equilibria \leq \dim \homol_1(G) - p + 1.
\]
\end{theorem}

Theorem~\ref{thm:intro_snake} allows to distinguish dynamics on graphs that are
homotopy equivalent, see Figure~\ref{fig:homotopic}.
While Theorem~\ref{thm:intro_snake} provides an upper bound on~$\dim \equilibria$
based on graph homology, in Section~\ref{sec:coverings}
we derive a lower bound based on graph coverings. 
Another structural result will relate automorphisms~$G\to G$ to
singularities of~$\equilibria$.

A second family of results concerns stability.
Dynamical systems of the form~\eqref{eq:main} exhibit both a conserved and
a gradient energy, which can be exploited to determine stability of equilibria.
For example, we prove that the \emph{block decomposition} of a graph
into its $2$-vertex-connected components, obtained by splitting the graph at its cut vertices,
induces a stability-preserving decomposition of the
set of equilibria:

\begin{theorem} \label{thm:intro_blocks_stability}
The block decomposition~$G = G_1 \cup \cdots \cup G_p$ induces a diffeomorphism
\[
	\argequilibria{G}{f} \cong \argequilibria{G_1}{f} \times \cdots \times \argequilibria{G_p}{f}
\]
that preserves stability.
More precisely, a point~$\x \in \R^{V(G)}$ is an equilibrium of~$G$
if and only if for every block~$G_k$ the point~$\x\vert_{V(G_k)} \in \R^{V(G_k)}$
is an equilibrium of~$G_k$,
and the statement remains true if ``equilibrium'' is replaced by ``Lyapunov stable equilibrium'',
or by ``asymptotically stable equilibrium up to translational symmetry'',
or by ``linearly stable equilibrium up to translational symmetry''.
\end{theorem}

This result allows to reduce analysis to $2$-vertex-connected graphs.
Other stability results in this paper concern \emph{topological bifurcation}, that is,
the change in stability along a manifold of equilibria.

\begin{figure}
\centering
\includegraphics[width=0.35\columnwidth]{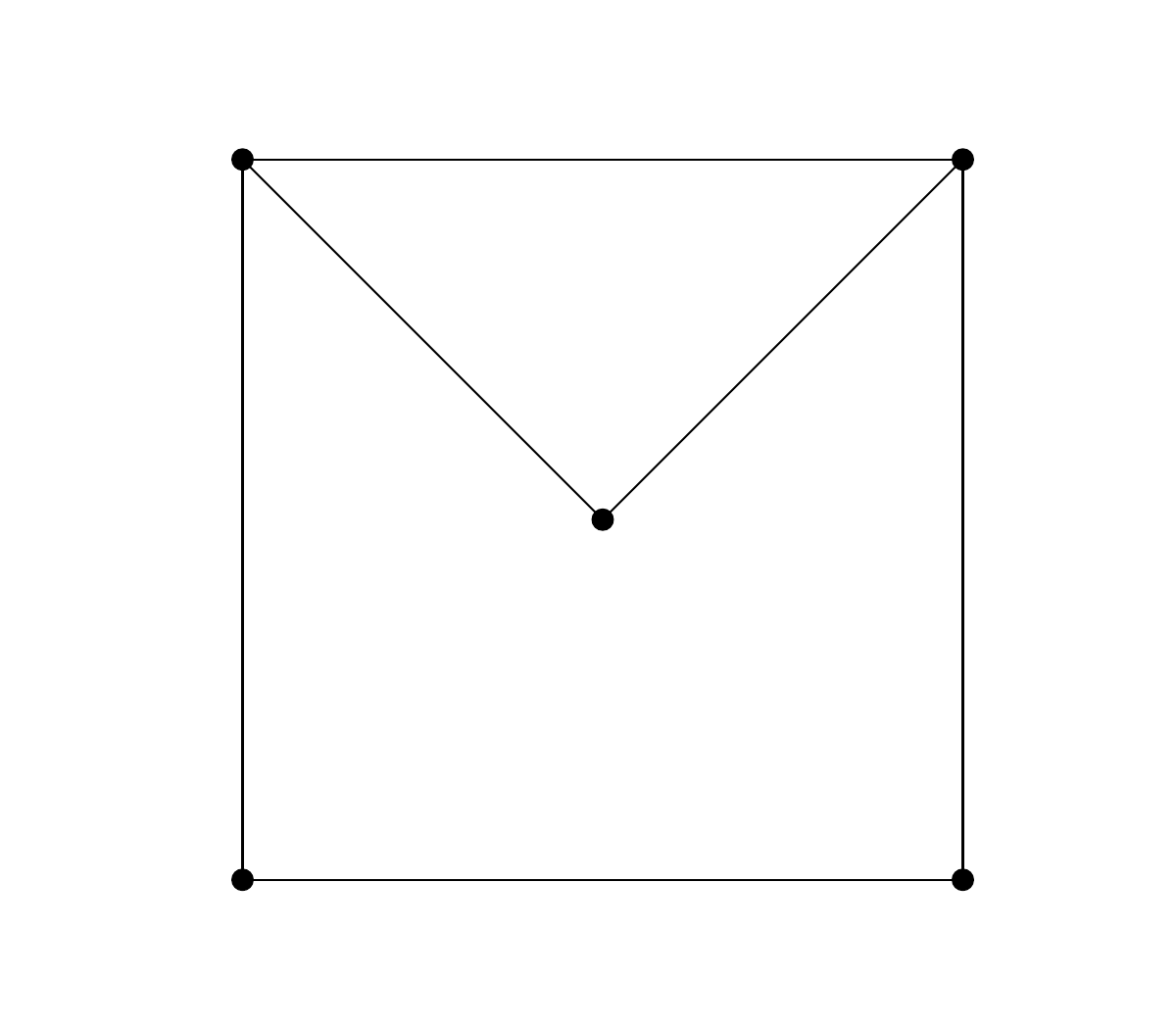}
\includegraphics[width=0.35\columnwidth]{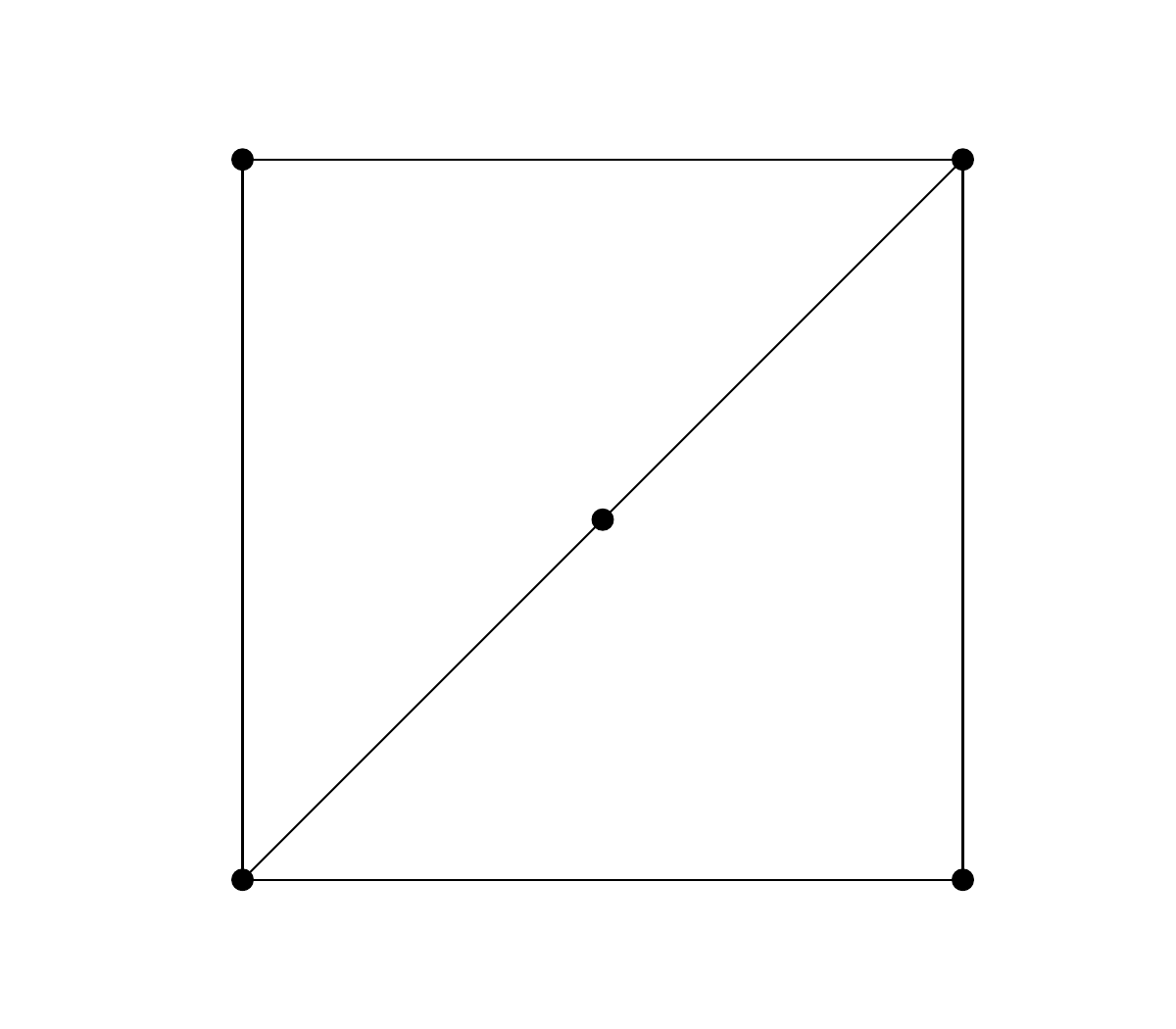}
\caption{
Although these graphs have the same number of vertices, edges, and the same homology,
the right hand side graph supports manifolds of equilibria of larger dimension.
}
\label{fig:homotopic}
\end{figure}

In recent years,
dynamical systems of the form~\eqref{eq:main} have independently emerged
across various research field, both within and beyond mathematics.
The linear case~$f(x)=x$, known as graph heat equation,
appears in
machine learning~\cite{kondor2002diffusion, belkin2003laplacian, wang2021dissecting, thanou2017learning},
control theory~\cite{medvedev2012stochastic, olfati2004consensus, olfati2006swarms},
image processing~\cite{zhang2008graph},
and electrical engineering~\cite{risi2002diffusion}.
The nonlinear case~$f(x)=\sin(x)$ is a classic topic in phase oscillator networks~\cite{brown2003globally, watanabe1997stability, ashwin2016identical, canale2009, Wiley2006, delabays2017multistability, lu2020, DeVille2016, sclosa2021completely, sclosa2022kuramoto}.
Finally, the case in which~$f$ is an odd polynomial is used to model polarization in
opinion dynamics~\cite{devriendt2021nonlinear, srivastava2011bifurcations, homs2021nonlinear}.

To further illustrate our results, we consider some concrete examples.

\begin{example} \label{ex:sine_complete}
Let~$K_n$ denote the complete graph with~$n\geq 3$ vertices and let~$f(x)=\sin(x)$.
Then the set~$\argequilibria{K_n}{f}$ is union of a family of isolated equilibrium points,
either stable or saddles, and a $(n-3)$-dimensional set
of unstable equilibria~\cite{brown2003globally, watanabe1997stability, mehta2015algebraic, ashwin2016identical}. The topology of this set is well known:
it corresponds to the configuration space of
a polygon linkage~\cite{kamiyama1996topology, kamiyama1992elementary, kapovich1996symplectic, mandini2014duistermaat}. For example, let~$n=4$.
Modulo~$2\pi$, the set~$\argequilibria{K_4}{f}$
is union of $4$~isolated saddle points, $1$~isolated stable equilibrium point,
and~$3$ mutually tangent circles of unstable equilibria~\cite[Figure~5]{sclosa2022kuramoto}.
\end{example}

\begin{example} \label{ex:poly_cycle}
Let~$C_n$ denote the cycle graph with~$n\geq 3$ vertices and let~$f$ be an odd polynomial
with~$n$ distinct real roots. Let~$x_1,\ldots,x_n$ denote the state of the vertices,
labelled in a circular order.
For every~$\lambda$ small enough the equation~$f(y)=\lambda$
has $n$~distinct roots~$y_1,\ldots,y_n$ and
an equilibrium can be obtained by
setting~$x_k = y_1+\cdots+y_{k}$~\cite[Proposition 4.4]{homs2021nonlinear}.
As~$\lambda$ varies, this gives a curve of equilibria.
For example, the set~$\argequilibria{C_3}{x^3-x}$
is union of a closed curve of stable
equilibria and an isolated unstable equilibrium point~\cite[Figure~2]{homs2021nonlinear}.
\end{example}

\begin{example} \label{ex:homotopic_I}
Let~$G_1$ and~$G_2$ denote the left hand side graph and the right hand side graph
of Figure~\ref{fig:homotopic} respectively.
Suppose that the function~$f$ is either periodic or a polynomial.
Then the set of equilibria~$\argequilibria{G_1}{f}$ is either $0$~or~$1$-dimensional,
while the set of equilibria~$\argequilibria{G_2}{f}$ can be $2$-dimensional.
The reason is given by Theorem~\ref{thm:intro_blocks_stability}:
the graph~$G_1$ contains two cycles intersecting in one edge,
while the graph~$G_2$ does not.
\end{example}

\begin{example} \label{ex:skew_book}
Let~$B_p$ be a \emph{triangular book} with~$p\geq 2$ \emph{pages}, that is,
the union of~$p$ triangles sharing a common edge, called~\emph{spine},
see Figure~\ref{fig:book_asymmetric}.
Suppose that there is~$P>0$
such that~$f(P+x)=-f(x)$ for every~$x \in \R$
(functions with this property are characterized by a
Fourier expansion of the form~$\sum_{m\, \text{odd}} a_m \sin(m\pi x/P)$).
Let the spine vertex states be equal to~$0$ and~$P$ and
denote the the states of the other vertices by~$x_1,\ldots,x_p$.
Then for every solution of the equation
$
	\sum_{k=1}^p f(x_k) = 0
$
the configuration~$(0,P,x_1,\ldots,x_p)$ is an equilibrium.
In particular~$\dim \argequilibria{B_p}{f} \geq p-1$.
It will follow from Theorem~\ref{thm:snake}
that~$\dim \argequilibria{B_p}{f} \leq p-1$,
thus~$\dim \argequilibria{B_p}{f} = p-1$.
\end{example}

\begin{figure}
\centering
\includegraphics[width=0.36\columnwidth]{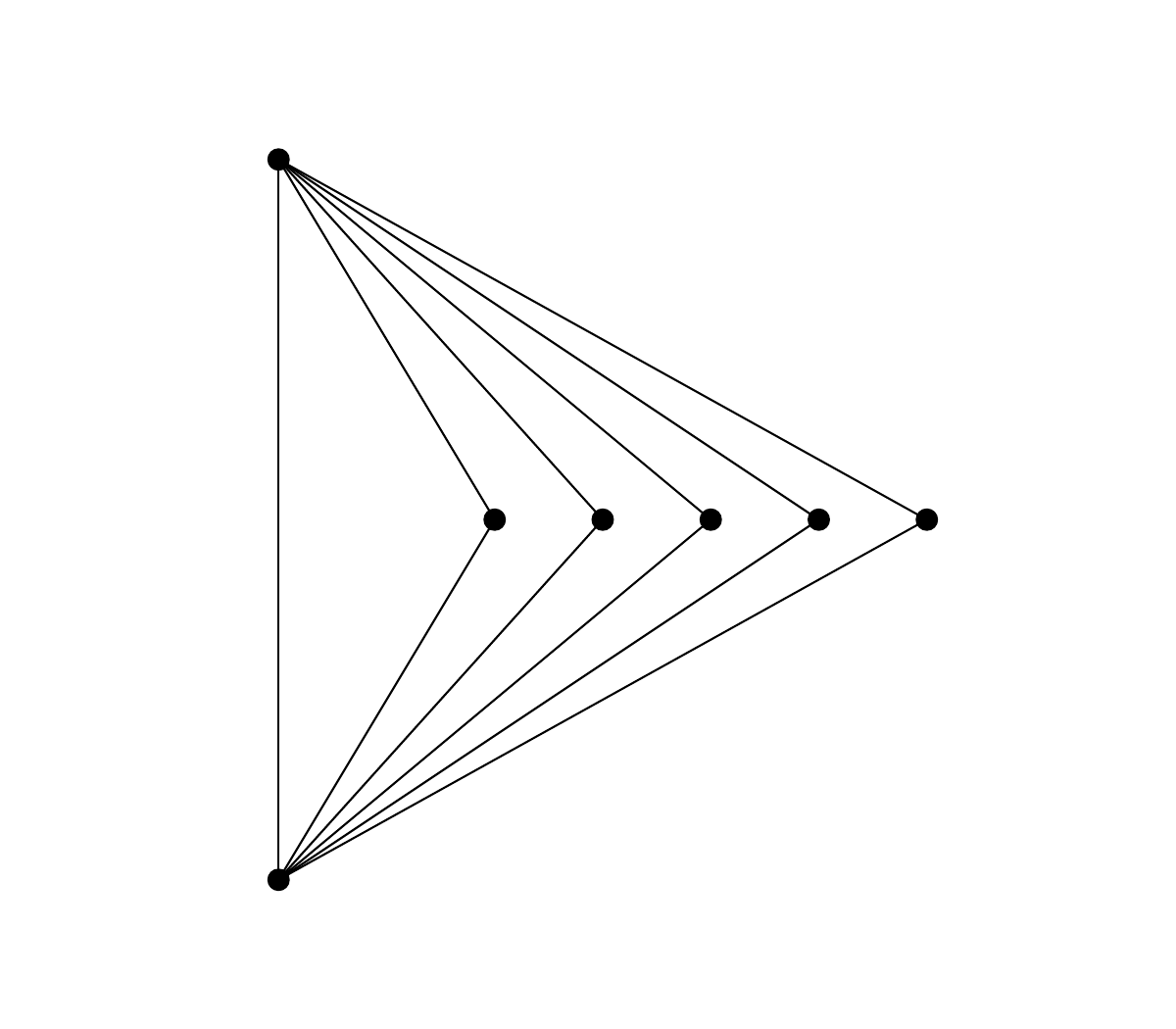}
\includegraphics[width=0.36\columnwidth]{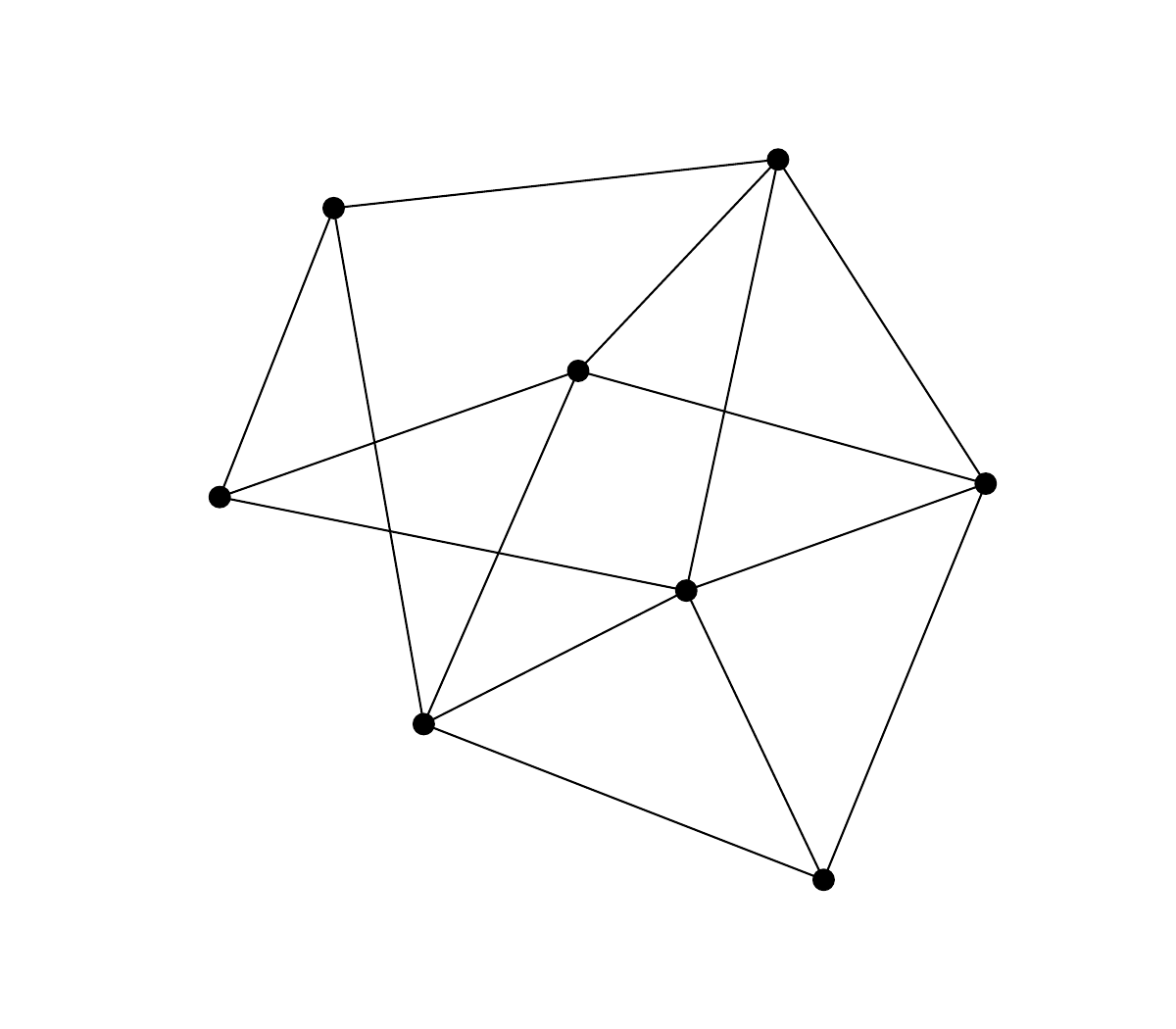}
\caption{
The triangular book with $5$~pages, on the left, supports a $4$-dimensional manifold of equilibria.
The graph on the right, although asymmetric, also
supports a continuum of equilibria.
}
\label{fig:book_asymmetric}
\end{figure}

The previous examples
might suggest that sets of equilibria of positive dimension can
only occur on very special, highly symmetric graphs. This is not the case:

\begin{example} \label{ex:asymmetric_I}
Let~$G$ denote the right hand side graph of Figure~\ref{fig:book_asymmetric}.
The graph~$G$ is \emph{asymmetric}: the only graph automorphism is the identity.
Moreover, the graph~$G$ is not planar, not regular, not bipartite, and not Eulerian.
However, it supports curves of equilibria.
Using generalized graph coverings,
we will exhibit a family of polynomials~$f$ for which~$\dim \equilibria \geq 1$,
see Example~\ref{ex:asymmetric_II}.
\end{example}

%%%%%%%%%%%%%%%%%%%%%%%%%%%%%%%%%%%%%%%%%%%%%%%%%%%%%%%%%%%%%%%%%%%%%%%%%%%%%%%%%%%%%%%%%%%%%%%%

\subsection{Structure of the Paper}
The paper is structured as follows:
in Section~\ref{sec:preliminaries} we review mostly-known facts about
systems of the form~\eqref{eq:main};
Sections~\ref{sec:set_of_equilibria} and~\ref{sec:stability} contain our original contributions;
Section~\ref{sec:discussion} provides final remarks and open questions.

%%%%%%%%%%%%%%%%%%%%%%%%%%%%%%%%%%%%%%%%%%%%%%%%%%%%%%%%%%%%%%%%%%%%%%%%%%%%%%%%%%%%%%%%%%%%%%%%
%%%%%%%%%%%%%%%%%%%%%%%%%%%%%%%%%%%%%%%%%%%%%%%%%%%%%%%%%%%%%%%%%%%%%%%%%%%%%%%%%%%%%%%%%%%%%%%%
%%%%%%%%%%%%%%%%%%%%%%%%%%%%%%%%%%%%%%%%%%%%%%%%%%%%%%%%%%%%%%%%%%%%%%%%%%%%%%%%%%%%%%%%%%%%%%%%

\section{Preliminaries} \label{sec:preliminaries}
In this section we discuss some elementary facts about
systems of the form~\eqref{eq:main}.

\begin{definition} \label{def:anal}
We denote by~$\anal$ the set of non-constant, odd, real-analytic
functions from~$\R$ to~$\R$.
\end{definition}
Since~$\anal$ contains~$f(x)=\sin(x)$ and every odd polynomial, this class covers
all the previously mentioned applications. Many of our results extend without modication
to the more general class of smooth functions with isolated critical points.

Fix a graph~$G$ and a function~$f\in \anal$.
Let~$n,m,c$ denote the number of vertices, edges and connected components of the graph~$G$.
Let~$V$ be the set of vertices.
Every edge of~$G$ can be oriented in two ways.
Fix an arbitrary orientation of each edge and let~$E \subseteq V\times V$ denote
the set of oriented edges.
If~$H$ is a subgraph of~$G$, we denote by~$V(K)$ and~$E(K)$
the set of vertices and oriented edges of~$K$.

Consider the coupled dynamical system~\eqref{eq:main}.
Analogous to Newton's third law of motion,
the input~$f(x_j - x_i)$ received by~$i$ from~$j$ is the opposite
of~$f(x_i - x_j)$, the input received by~$j$ from~$i$.
As a consequence:

\begin{lemma} \label{lem:constant_of_motion}
For every connected component~$K$ of~$G$ we have
$
	\sum_{i \in K} \dot x_i = 0.
$
In particular the quantity~$\sum_{i\in K} x_i$ is a constant of motion of the system.
\end{lemma}
\begin{proof}
Since~$f$ is odd we have
\[
	\sum_{i \in K} \dot x_i = \sum_{(j,k)\in E(K)} f(x_j-x_k) - f(x_k-x_j) = 0.
\]
This completes the proof.
\end{proof}

Lemma~\ref{lem:constant_of_motion} shows that the dynamical system~\eqref{eq:main}
has~$c$ independent conserved quantities~$\sum_{i\in K} x_i$, one for each connected component~$K$.
Therefore, one can interpret dynamics as a diffusive process on the graph vertices.

To proceed further, we will need the notion of graph homology~\cite{diestel, hatcher}.

\begin{definition}
Consider the linear map~$B: \edgespace\to\vertexspace$, associated to the graph~$G$,
induced by the matrix~$B$ with entries~$b_{i, (j,k)} = \delta_{ik}-\delta_{ij}$,
where~$\delta$ is the Kronecker delta.
The matrix~$B$ is called \emph{incidence matrix}.
The subspaces~$\homol_0(G) = \ker (B^\trans)$ and~$\homol_1(G) = \ker (B)$
are called~\emph{zeroth}
and~\emph{first homology group} respectively.
\end{definition}

When no confusion arises, we will drop dependence on~$G$ and simply
denote~$\homol_0$ and~$\homol_1$.
The following lemma describes the gradient structure of the system.

\begin{lemma} \label{lem:energy}
Let~$g$ be any primitive of~$f$. Then the system~\eqref{eq:main}
can be written as~$\dot \x = -\nabla \E(\x)$
where~$\E: \vertexspace \to \R$ is defined by
\begin{equation} \label{eq:energy}
	\E(\x) = \sum_{(j,k) \in E} g(x_k - x_j).
\end{equation}
\end{lemma}
\begin{proof}
We have
\[
	\partial_i \E(\x) = \sum_{(j,k) \in E} b_{i, (j,k)} f(x_k - x_j)
		= - \sum_{j\in N(i)} f(x_j - x_i) = \dot x_i.
\]
where the second equality follows since~$f$ is odd.
\end{proof}

\begin{corollary} \label{cor:vector_form}
Define the function~$\f: \edgespace \to \edgespace$ component-wise as
\[
	\f (y_1,\ldots, y_m) = (f(y_1),\ldots,f(y_m)).
\]
Then the system~\eqref{eq:main} can be written
as~$
	\dot \x = - B \f(B^\trans \x).
$
\end{corollary}
\begin{proof}
Let~$\1 \in \edgespace$ denote vector with all entries equal to~$1$.
Let~$\g: \edgespace \to \edgespace$ be defined component-wise by~$g$.
Then~$\E(\x) = - \1^\trans B g (B^\trans \x)$.
Taking the gradient concludes the proof.
\end{proof}

The zeroth homology group~$\homol_0$
is generated by the indicator vectors~$\d_1$, \ldots, $\d_c \in \vertexspace$
of the connected components~$K_1$, \ldots, $K_c$ of~$G$:
the $i$-th entry of the vector~$\d_j$ is equal to~$1$ if the vertex~$i$ belongs to~$K_j$
and~$0$ otherwise.
The zeroth homology group~$\homol_0$ acts on~$\R^V$ by translation.
The action makes~$\homol_0$ a group of symmetries of the dynamical system:

\begin{proposition} \label{prop:sym}
Consider the orthogonal decomposition~$\vertexspace = \bigcup_{\d\in \homol_0} \d + \homol_0^\perp$.
For every~$\d\in \homol_0$ the affine subspace~$\d + \homol_0^\perp$ is dynamically invariant
and the bijection
\[
	\homol_0^\perp \to \d + \homol_0^\perp, \quad \x \mapsto \d+\x
\]
maps trajectories to trajectories.
\end{proposition}
\begin{proof}
By Lemma~\ref{lem:constant_of_motion} each subspace~$\d + \homol_0^\perp$
is dynamically invariant.
By Corollary~\ref{cor:vector_form} the vector field
is invariant under translations by elements of~$\homol_0 = \ker B^\trans$.
\end{proof}

Proposition~\ref{prop:sym} suggests to consider quotient dynamics.
Moreover, it shows that quotient dynamics can by obtained
by restricting the system to any leaf~$\d + \homol_0^\perp$ of the foliation.
However, a canonical choice for representing quotient dynamics is obtained
by transforming the system from vertex space to edge space:
letting~$\y = B^\trans \x$ gives
\begin{equation} \label{eq:main_edgespace}
    \begin{cases} 
    \dot \y = - B^\trans B \f (\y)\\
    \y \in \homol_1^\perp.
    \end{cases}
\end{equation}
The system~\eqref{eq:main_edgespace} represents the dynamics of~\eqref{eq:main} up to symmetry.
The condition~$\y \in \homol_1^\perp$ guarantees that~$\y = B^\trans \x$ for some~$\x$.
As independently noticed in~\cite{homs2021nonlinear},
the advantage of working in edge space is a simple expression for the set of
equilibria:
\begin{equation} \label{eq:intersection}
	\f^{-1}(\homol_1) \cap \homol_1^\perp.
\end{equation}

The first homology group~$\homol_1$
is generated by the oriented cycles of~$G$,
encoded as vectors in~$\edgespace$
with entries~$+1$,~$-1$, or~$0$, according to whether an oriented edge
agrees with the orientation of the cycle, has opposite orientation, or does
not belong to the cycle, respectively.
Therefore, equation~\eqref{eq:intersection} suggests a connection between equilibria
of the system and cycles in the underlying graph.
This idea also appears in~\cite{delabays2017multistability, sclosa2021completely}
and will be further explored in the following section.

The decomposition of a connected graph by the set of its cut-vertices,
known as~\emph{block decomposition} or~\emph{block tree decomposition},
is a basic concept in graph theory~\cite{engelking1978dimension, tutte1966connectivity}.
A \emph{block} of a graph is a subgraph which is connected, has no cut vertices,
and is maximal with these properties.
A key feature of the system~\eqref{eq:main} is that equilibria of a graph can be decomposed
into the equilibria of the blocks:

\begin{proposition} \label{prop:blocks}
Let
$
	G = G_1 \cup \cdots \cup G_p
$
be the block composition of a connected graph~$G$.
Then a point~$\x \in \R^{V(G)}$ is an equilibrium of the graph~$G$
if and only if for every~$k=1,\ldots,p$ the restriction~$\x\vert_{V(G_k)} \in \R^{V(G_k)}$
is an equilibrium of the graph~$G_k$.
\end{proposition}
\begin{proof}
See~\cite{canale2007gluing} for~$f(x)=\sin(x)$
and~\cite[Section 3.5]{homs2021nonlinear} for the general case.
\end{proof}

%%%%%%%%%%%%%%%%%%%%%%%%%%%%%%%%%%%%%%%%%%%%%%%%%%%%%%%%%%%%%%%%%%%%%%%%%%%%%%%%%%%%%%%%%%%%%%%%
%%%%%%%%%%%%%%%%%%%%%%%%%%%%%%%%%%%%%%%%%%%%%%%%%%%%%%%%%%%%%%%%%%%%%%%%%%%%%%%%%%%%%%%%%%%%%%%%
%%%%%%%%%%%%%%%%%%%%%%%%%%%%%%%%%%%%%%%%%%%%%%%%%%%%%%%%%%%%%%%%%%%%%%%%%%%%%%%%%%%%%%%%%%%%%%%%

\section{The Set of Equilibria} \label{sec:set_of_equilibria}
In this section we analyze the structure of the set of equilibria, in relation
to the combinatorial properties of the underlying graph and the shape of the coupling function.
As shown in the previous section, the dynamical system~\eqref{eq:main}, which takes
place in the vertex space~$\vertexspace$, can be represented in edge space~$\edgespace$.
This has the advantage of getting rid of translational symmetry.
We identify the set of equilibria up to symmetry with the set of equilibria in edge space.

\begin{definition} \label{def:set_of_equilibria}
Let~$G$ be a finite graph and let~$f\in \anal$. Define
\[
	\equilibria = \f^{-1}(\homol_1(G)) \cap \homol_1(G)^\perp.
\]
\end{definition}

Every point of~$\equilibria$ is an equilibrium in~$\edgespace$
and corresponds to a $c$-dimensional affine set of
of equivalent equilibria in~$\vertexspace$.

%%%%%%%%%%%%%%%%%%%%%%%%%%%%%%%%%%%%%%%%%%%%%%%%%%%%%%%%%%%%%%%%%%%%%%%%%%%%%%%%%%%%%%%%%%%%%%%%
\subsection{Real-Analytic Sets}
The set~$\equilibria$ is the solution set of a finite system of real-analytic equations.
A set with this property is known as \emph{real-analytic set}~\cite{massey2007notes, whitney1972complex, narasimhan2006introduction}, or simply \emph{analytic sets}.
Let us recall some known results.

Let~$X$ be an analytic set.
Then in a neighborhood of every point~$x\in X$
the set~$X$ is union of finitely many
smooth manifolds~\cite[Theorem 1.2.10]{trotman2020stratification}.
These manifold are not necessarily closed and not necessarily of the same dimension.
For example, the subset of~$\R^2$ given by~$xy=0$ is union of five manifolds:
four smooth curves and a point.

A point~$x\in X$ is called \emph{regular}
if there is a neighborhood of~$x$ in~$X$ which is locally homeomorphic
to a smooth manifold. A point is called~\emph{singular} otherwise.
The \emph{dimension} of~$X$, denoted by~$\dim X$, is the maximum dimension of a manifold
contained in~$X$.
Singular points form a nowhere dense subset of~$X$~\cite[Theorem 5.14]{massey2007notes}.

Notice how these general topological facts agree with
the sets of equilibria found in Example~\ref{ex:sine_complete},
Example~\ref{ex:poly_cycle} and Example~\ref{ex:skew_book}.

%%%%%%%%%%%%%%%%%%%%%%%%%%%%%%%%%%%%%%%%%%%%%%%%%%%%%%%%%%%%%%%%%%%%%%%%%%%%%%%%%%%%%%%%%%%%%%%%

\subsection{Dimension of the Set of Equilibria}
We will need the following result:

\begin{lemma} \label{lem:dimension}
Let~$G$ be a finite graph and let~$f\in \anal$. Then
\begin{equation} \label{eq:lem_dimension}
	\dim(\equilibria) =
	\dim \p{\f^{-1}(\homol_1) \cap \homol_1^\perp} \leq \dim \p{\homol_1\cap \f(\homol_1^\perp)}.
\end{equation}
\end{lemma}
\begin{proof}
Since~$f$ is analytic and not constant, the critical points of~$f$ are isolated.
Therefore the space~$\edgespace \simeq \R^m$ is union of countably many
compact sets~$V_k$ of the form
\[
	V_k = \prod_{j=1}^m [\alpha_{k,j},\beta_{k,j}]
\]
such that the restriction
$
	\f\vert_{V_k}: V_k \to \f(V_k)
$
is an homeomorphism. We claim that
\begin{align}
	\dim(\f^{-1}(\homol_1) \cap \homol_1^\perp)
	& = \sup_k \dim \p{\f\vert_{V_k}^{-1}(\homol_1) \cap V_k \cap \homol_1^\perp}
		\label{lem:dimension_1} \\
	& = \sup_k \dim \p{\homol_1 \cap \f\vert_{V_k}(V_k \cap \homol_1^\perp)} \label{lem:dimension_2} \\
	& \leq \sup_k \dim \p{\homol_1 \cap \f\vert_{V_k}(V_k) \cap \f(\homol_1^\perp)} \\
	& = \dim \p{\homol_1\cap \f(\homol_1^\perp)} \label{lem:dimension_3}.
\end{align}
Notice that~$V_k \cap \f^{-1}(\homol_1) \cap \homol_1^\perp$ is closed in~$\f^{-1}(\homol_1) \cap \homol_1^\perp$
and~$\f\vert_{V_k}(V_k) \cap \homol_1\cap \f(\homol_1^\perp)$ is compact,
thus closed in~$\homol_1\cap \f(\homol_1^\perp)$.
By~\cite[Theorem~5.4]{keesling1986hausdorff} the dimension of a countable union of closed sets
is the supremum of the dimensions. This proves~\eqref{lem:dimension_1} and~\eqref{lem:dimension_3}.
Equation~\eqref{lem:dimension_2} follows
since~$\f\vert_{V_k}: V_k\cap \homol_1^\perp \to \f\vert_{V_k}(V_k \cap \homol_1^\perp)$
is an homeomorphism.
\end{proof}

As a consequence, we obtain that the dimension of the set of equilibria~$\equilibria$
can be bounded in terms of number cycles in the underlying graph:

\begin{theorem} \label{thm:dimension}
Let~$G$ be a finite graph and let~$f\in \anal$. Then
\begin{equation} \label{eq:thm_dimension}
	\dim \equilibria \leq \dim \homol_1(G).
\end{equation}
\end{theorem}
\begin{proof}
By definition~$\equilibria = \f^{-1}(\homol_1(G)) \cap \homol_1(G)^\perp$.
Lemma~\ref{lem:dimension} implies
\[
	\dim \equilibria \leq \dim \p{\homol_1(G)\cap \f(\homol_1(G)^\perp)}.
\]
Since~$\homol_1(G)\cap \f(\homol_1(G)^\perp) \subseteq \homol_1(G)$ and
dimension is monotone with
respect to inclusion, we conclude that~$\dim \equilibria\leq \dim \homol_1(G)$.
\end{proof}

As a corollary, we obtain the following upper bounds:

\begin{corollary}
Let~$G$ be a finite graph and let~$f\in \anal$. Then
\[
	\dim \equilibria \leq n-c, \quad
	\dim \equilibria \leq \floor{\frac{m}{2}}, \quad
	\dim \equilibria \leq m-n+c.
\]
\end{corollary}
\begin{proof}
The upper bound~$n-c$ follows from~$\equilibria \subseteq \homol_1(G)^\perp$.
By~\cite[Theorem 1.9.5]{diestel} we have~$\dim \homol_1(G) = m-n+c$.
Theorem~\ref{thm:dimension} implies~$\dim \equilibria\leq m-n+c$.
Since the positive integers~$n-c$ and~$m-n+c$ have sum~$m$,
at lest one is not larger than~$m/2$, thus~$\dim \equilibria \leq m/2$.
\end{proof}

For planar graphs we obtain:

\begin{corollary} \label{cor:equilibria_planar}
Let~$G$ be a planar graph with~$r$ bounded planar regions and let~$f\in \anal$. Then
$
	\dim \equilibria \leq r.
$
\end{corollary}
\begin{proof}
By Euler's formula~$n-m+r=1$. Therefore~$\dim \homol_1 = m-n+c=r$.
\end{proof}

The upper bound~$\dim \equilibria \leq \dim \homol_1(G)$ is optimal
for trees and cycle graphs, for which~$\dim \homol_1(G)\leq 1$,
but can be unboundedly bad for~$\homol_1(G)\geq 2$:
We will see that on the family of wheel graphs
the difference~$\dim \homol_1(G) - \equilibria$ is unbounded.
The reason is that on graphs containing more than one cycle,
not only not the number of cycles is relevant,
but also how cycles intersect.
To explain this phenomenon, we introduce the following combinatorial notion:

\begin{definition} \label{def:snake}
A \emph{cycle chain of length}~$p$ is a sequence of cycle graphs~$G_1$, \ldots, $G_p$
such that~$\abs{E(G_j) \cap E(G_{k})}=1$ if~$\abs{j-k}=1$
and~$E(G_j)\cap E(G_k) = \emptyset$ if~$\abs{j-k}>1$.
\end{definition}

The proof of the following theorem requires a technical hypothesis:
the coupling function~$f$ is either periodic or has \emph{finite fibers},
that is, for every~$x\in \R$ the preimage~$f^{-1}(\{x\})$ is a finite set.
Note that every non-constant polynomial has finite fibers.

\begin{theorem} \label{thm:snake}
Let~$G$ be a finite graph and
let~$\cc(G)$ denote the maximum length of a cycle chain contained in~$G$.
Suppose that~$f\in \anal$ is either periodic or has finite fibers. Then
\[
	\dim \equilibria \leq \dim \homol_1(G) - \cc(G) + 1.
\]
\end{theorem}

\begin{proof}
We write the proof in the case of finite fibers.
The proof can be adapted to the case of periodic~$f$
by reducing the space modulo periodicity and noting that this reduction
preserved dimensions.

Let~$p=\cc(G)$. By hypothesis, the graph~$G$ contains~$p$ cycle subgraphs
forming a cycle chain. Fix an orientation of each of these cycles
and let~$\ggamma_1,\ldots,\ggamma_p$ be the corresponding vectors in~$\edgespace$.
Recall that cycles are encoded as vectors as follows:
the coefficient~$\gamma_{k,e}$ is equal to~$+1,-1,0$ according to whether
the oriented edge~$e$ agrees with the orientation of the cycle, has the opposite orientation
or does not belong to the cycle.
By abuse of notation we denote by~$E(\ggamma_k)$ the set of edges of the cycle
encoded by~$\ggamma_k$.
By definition of cycle chain it follows that the vectors~$\ggamma_1,\ldots,\ggamma_p$
are linearly independent. Complete them to a
basis~$\ggamma_1,\ldots,\ggamma_p, \ddelta_{p+1}, \ldots, \ddelta_{\dim \homol_1}$ of~$\homol_1$
whose elements are oriented cycles.

Let~$\y\in \equilibria$. Since~$\f(\y) \in \homol_1$ there
are~$z_1,\ldots,z_{\dim \homol_1} \in \R$ such that
\begin{equation} \label{eq:snake_fy}
	\f(\y) = z_1 \ggamma_1 + \ldots z_p \ggamma_p
		+ z_{p+1} \ddelta_{p+1} \ldots + z_{\dim \homol_1} \ddelta_{\dim \homol_1}.
\end{equation}
The idea of the proof is showing that for every~$k<p$ the coefficient~$z_k$ can be expressed
in terms of~$z_{k+1},\ldots,z_{\dim \homol_1}$.

Fix~$k<p$. Let~$e$ be the only edge
in the intersection~$E(\ggamma_k) \cap E(\ggamma_{k+1})$.
Since~$\ggamma_{k-1}$ and~$\ggamma_{k+1}$ are the only cycles in the chain
intersecting~$\ggamma_k$, from~\eqref{eq:snake_fy} it follows that
\begin{equation}
f(y_e) = z_k \gamma_{k,e} + z_{k+1} \gamma_{k+1,e}
	+ \sum_{i=p+1}^{\dim \homol_1} z_{i} \delta_{i,e}, \label{eq:snake_fe}
\end{equation}
Since~$e\in E(\ggamma_k)$ we have~$\gamma_{k,e} \neq 0$,
thus~$z_k$ can be expressed as a linear combination of~$z_{k+1},\ldots,z_{\dim \homol_1}$
and~$f(y_e)$.

We would like to express~$f(y_e)$ as a function of~$z_{k+1},\ldots,z_{\dim \homol_1}$.
In general this is impossible since~$f$ might not be invertible.
However, since~$f$ has finite fibers, we can show that for
fixed~$z_{k+1},\ldots,z_{\dim \homol_1}$ there are only
finitely many possibilities for~$f(y_e)$.
Since~$\y\in \homol_1^\perp$, in particular~$\y^\trans \ggamma_{k+1} = 0$, thus
\begin{equation} \label{eq:snake_intersect}
	 f(y_{e}) = f\p{- \gamma_{k+1,e}^{-1}
		\sum_{e' \in E(\ggamma_{k+1})\setminus \{e\}} \gamma_{k+1,e'} y_{e'}}.
\end{equation}
From~\eqref{eq:snake_fy} it follows that
for all~$e' \in E(\ggamma_{k+1})\setminus \{e\}$ the quantity~$f(y_{e'})$
is a linear combination of~$z_{k+1}$, \ldots, $z_{\dim \homol_1}$ with coefficients~$0,+1,-1$.
There are finitely many such linear combinations. Therefore~$y_{e'}$ lies in the $f$-preimage 
of a finite set, which by hypothesis is a finite set.
Therefore, from~\eqref{eq:snake_intersect} we
conclude that for fixed~$z_{k+1},\ldots,z_{\dim \homol_1}$
there are only finitely many choices for~$f(y_e)$, and thus from~\eqref{eq:snake_fe}
only finitely many choices for~$z_k$.

Let~$W \subset \edgespace$ be the vector subspace generated
by~$\ggamma_p,\ddelta_{p+1},\ldots,\ddelta_{\dim \homol_1}$.
Since the previous argument applies to every~$k<p$, it follows that the projection map
\begin{align*}
	\pi: &\homol_1\cap \f(\homol_1^\perp) \to W,
		\ (z_1,\ldots,z_{\dim \homol_1}) \mapsto (z_p,z_{p+1},\ldots,z_{\dim \homol_1})
\end{align*}
has finite fibers.
By~\cite[Proposition 4.11]{pears1975dimension} it follows that
the dimension of the domain~$\dim(\homol_1\cap \f(\homol_1^\perp))$ is at most
equal to the dimension of the codomain~$\dim W = \dim \homol_1-p+1$.
By Lemma~\ref{lem:dimension}, this concludes the proof.
\end{proof}

Clearly, Theorem~\ref{thm:snake} implies Theorem~\ref{thm:intro_snake}.
We can now justify some claims made in the examples in the introduction:

\begin{example}
In Example~\ref{ex:skew_book} we constructed a $(p-1)$-dimensional set of equilibria
on a triangular book graph with~$p\geq 2$ pages, showing that~$\dim \equilibria \geq p-1$.
Notice that the first two pages of a book form a cycle chain of length~$2$, thus by
Theorem~\ref{thm:snake}
\[
	\dim \equilibria \leq \dim \homol_1(G) - 2 + 1 = p-1.
\]
We conclude that~$\dim \equilibria = p-1$.
\end{example}

\begin{example} \label{ex:homotopic_II}
Let~$G$ be the left hand side graph of Figure~\ref{fig:homotopic} and let
the function~$f$ be either periodic or polynomial.
Notice that~$G$ contains two cycles with one edge in common (the edge on the top).
By Theorem~\ref{thm:snake} we conclude that~$\dim \equilibria \leq 1$.

Let~$G$ be the right hand side graph of Figure~\ref{fig:homotopic}. Consider the periodic 
function~$f(x)=\sin(x)$.
We claim that~$\dim \equilibria = 2$. Let the state of the vertices of degree~$3$ be equal to~$0$
and~$\pi$. Denote the states of the remaining vertices
by~$x_1,x_2,x_3$. Then every solution of the equation~$\sum_{k=1}^3 \sin(x_k) = 0$
gives an equilibrium. Therefore~$\dim \equilibria \geq 2$. Since~$\dim \homol_1(G) = 2$,
by Theorem~\ref{thm:dimension} we conclude that~$\dim \equilibria = 2$.
\end{example}

\begin{figure}
\includegraphics[width=0.35\columnwidth]{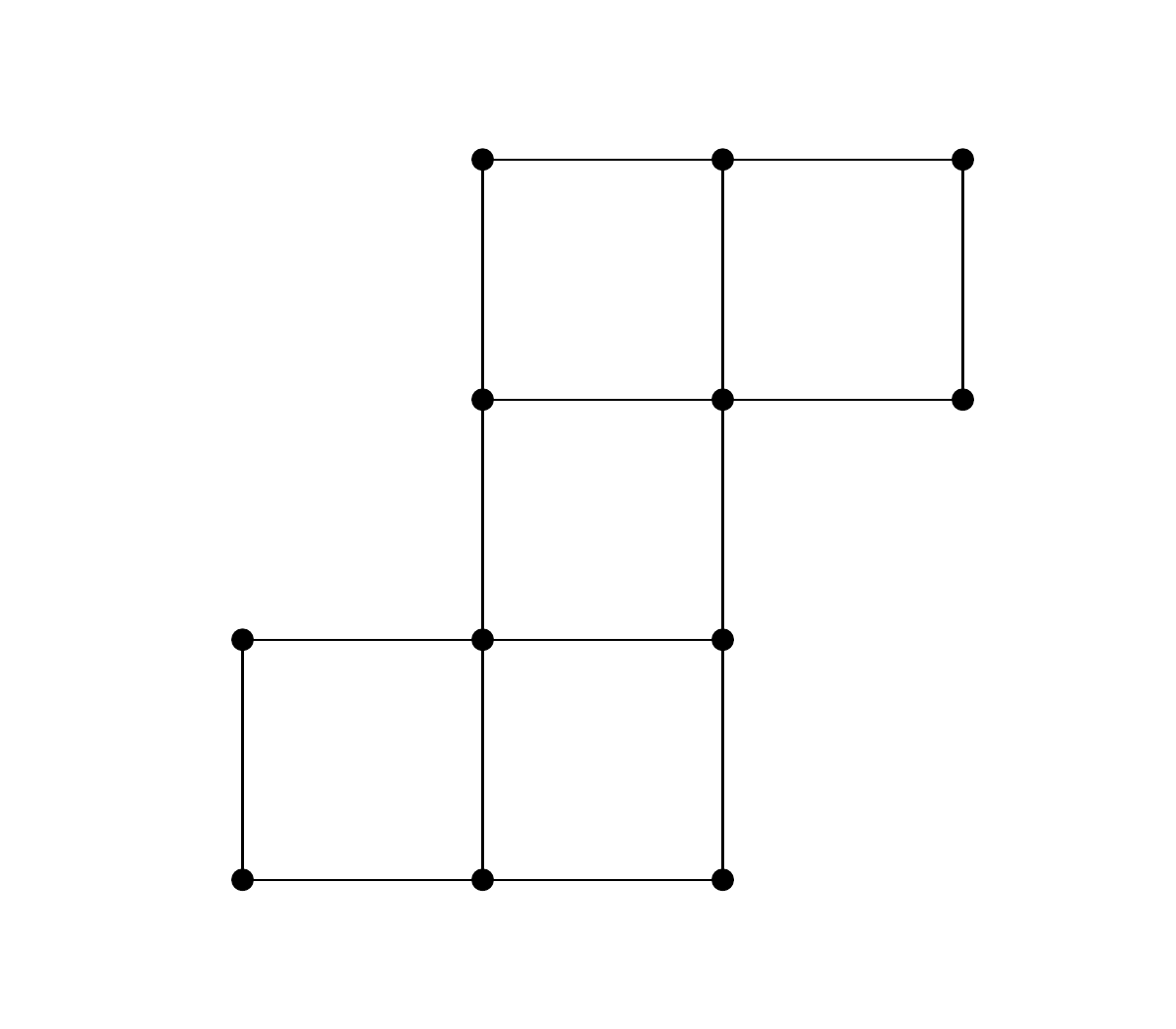}
\includegraphics[width=0.35\columnwidth]{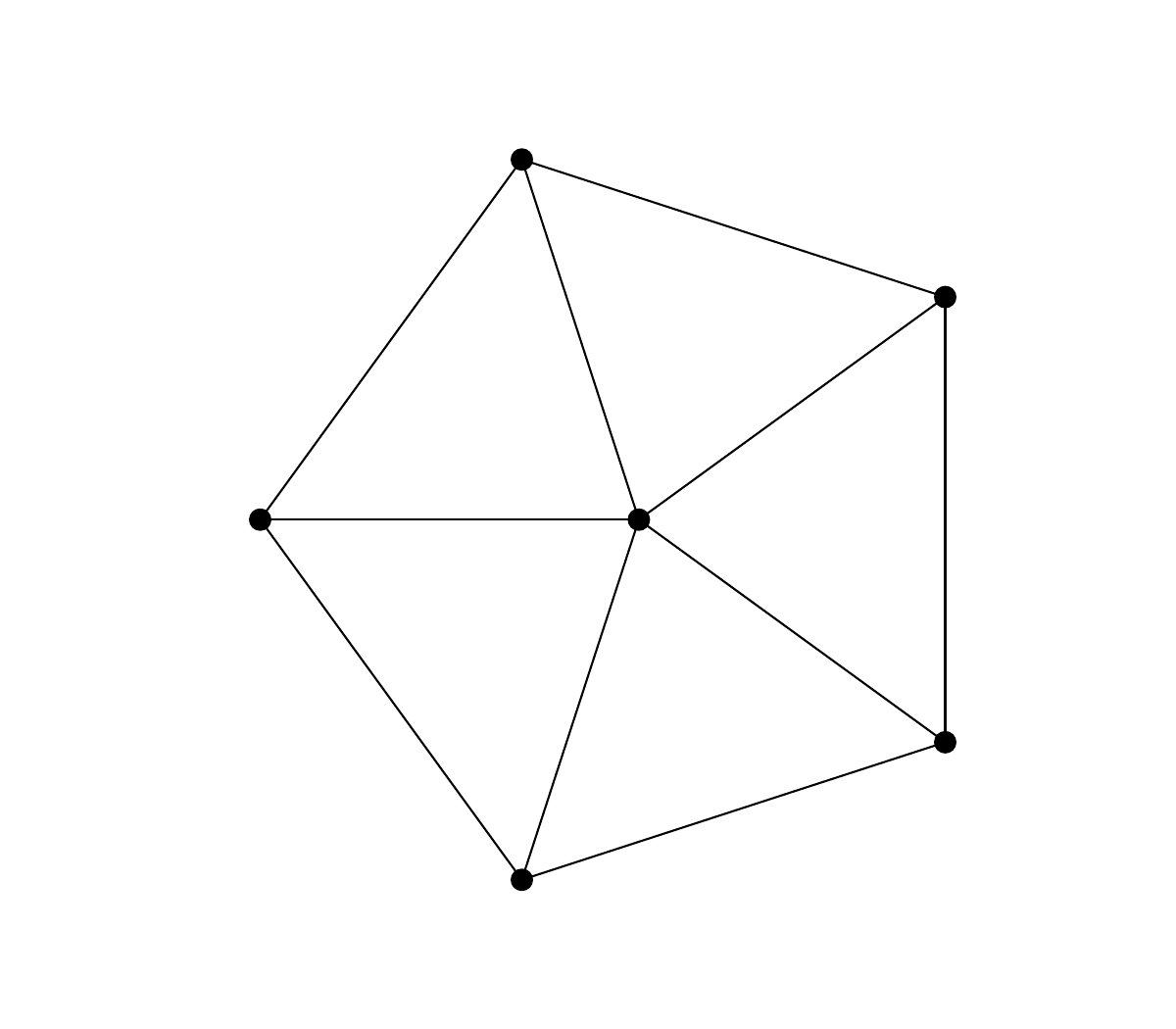}
\caption{A snake graph and a wheel graph.}
\label{fig:snake_wheel}
\end{figure}

Let us give two more examples:

\begin{example}
Snake graphs, which play a central role in cluster algebras~\cite{canakci2013snake},
are particular cases of cycle chains, see Figure~\ref{fig:snake_wheel}.
By Theorem~\ref{thm:snake}, if~$G$ is a snake graph and
the function~$f$ is polynomial or periodic, then~$\dim \equilibria \leq 1$.
\end{example}

\begin{example}
The \emph{wheel graph}~$W_n$ with~$n$ vertices is obtained by connecting the
vertices of the cycle graph~$C_{n-1}$ to a new vertex,
see Figure~\ref{fig:snake_wheel}.
We have~$\dim \homol_1(W_n)=n$ while~$\cc(W_n)=n-1$.
By Theorem~\ref{thm:snake}, if the function~$f$ is polynomial or periodic,
then~$\dim \argequilibria{W_n}{f} \leq 2$.
\end{example}

%%%%%%%%%%%%%%%%%%%%%%%%%%%%%%%%%%%%%%%%%%%%%%%%%%%%%%%%%%%%%%%%%%%%%%%%%%%%%%%%%%%%%%%%%%%%%%%%

\subsection{Coverings and Colorings} \label{sec:coverings}

The notion of graph covering is analogous to the notion of topological covering
and captures the idea that, locally, the covering graph looks like the covered one.
The main idea of this section is that the equilibria on the covered graph
can be lifted to equilibria on the covering graph.

\begin{definition} \label{def:covering}
Let~$G$ and~$H$ be two graphs and~$\varphi: V(G) \to V(H)$ a surjection.
Then~$\varphi$ is called \emph{covering map} if for every vertex~$i \in V(G)$
the restriction
\[
	\varphi\vert_{N_G(i)}: N_G(i) \to N_H(\varphi(i))
\]
is well-defined and bijective. If~$\varphi: V(G) \to V(H)$ is a covering map,
we say that the graph~$G$ is a \emph{covering} of the graph~$H$.
\end{definition}

Notice that if~$G$ is a covering of the complete graph~$K_n$, then
the preimages~$(\varphi^{-1}(j))_j$ define an $n$-coloring of~$G$:
any two adjacent vertices have different $\varphi$-image.
We generalize the notion of covering map as follows:

\begin{definition} \label{def:generalized_covering}
Let~$d\geq 1$ be an integer. We say that a map is \emph{$d$-to-one}
if every point of the codomain has exactly $d$~preimages in the domain.
Let~$G$ and~$H$ be two graphs and~$\varphi: V(G) \to V(H)$ a surjection.
Then~$\varphi$ is called \emph{generalized covering map} if for every vertex~$i \in V(G)$
there is an integer~$d_i\geq 1$ such that the restriction
\[
	\varphi\vert_{N_G(i) \setminus \varphi^{-1}(\varphi(i))}:
		N_G(i)\setminus \varphi^{-1}(\varphi(i)) \to N_H(\varphi(i))
\]
is well-defined and $d_i$-to-one. If~$\varphi: V(G) \to V(H)$ is a generalized covering map,
we say that the graph~$G$ is a \emph{generalized covering} of the graph~$H$.
\end{definition}

If~$d_i=1$ and~$N_G(i) \cap \varphi^{-1}(\varphi(i)) = \emptyset$ for every vertex~$i\in V(G)$,
we recover the notion of covering map of Definition~\ref{def:covering}.
If the integers~$d_i$ are all equal, but the sets~$N_G(i) \setminus \varphi^{-1}(\{i\})$
are not necessarily empty, we recover the notion of
external equitable partition~\cite{schaub2016graph}.
Figure~\ref{fig:asymmetric_covering} shows an example of generalized covering
that is not a covering nor a external equitable partition.

The following result shows that generalized coverings induce, in a contravariant way,
an embedding of the set of equilibria.

\begin{theorem} \label{thm:covering}
Let $G$~and~$H$ be finite graphs and let~$\varphi: V(G) \to V(H)$ be a generalized
covering map.
Then the map~$\varphi$ induces an
embedding
\[
	\varphi^*: \argequilibria{H}{f} \to \argequilibria{G}{f},
	\quad (\varphi^*(\y))_i = y_{\varphi(i)}
\]
of the set of equilibria up to symmetry of~$H$ into the set of equilibria up to
symmetry of~$G$.
In particular~$\dim \argequilibria{G}{f} \geq \dim \argequilibria{H}{f}$.
\end{theorem}
\begin{proof}
Fix an equilibrium~$\y \in \R^{V(H)}$. Let~$\x = \varphi^* (\y) \in \R^{V(G)}$
be defined as follows: for every vertex~$i \in V(G)$ let~$x_i = y_{\varphi(i)}$.
Fix a vertex~$i \in V(G)$ and let~$d_i\geq 1$ the integer
given by Definition~\ref{def:generalized_covering}. Then
\begin{align*}
	\dot x_i & = \sum_{j\in N_G(i)} f(x_j-x_i)
		= \sum_{j\in N_G(i)} f(y_{\varphi(j)}-y_{\varphi(i)}) 
			&& \text{by definition of } \x \\
		& = \sum_{j\in N_G(i)\setminus \varphi^{-1}(\{i\})} f(y_{\varphi(j)}-y_{\varphi(i)})
			&& \text{since } f(0)=0\\
		& = d_i \sum_{j\in N_H(\varphi(i))} f(y_j-y_{\varphi(i)}) = 0
			&& \text{since $\y$ is an equilibrium}.
\end{align*}
This shows that the point~$\x = \varphi^* (\y)$ is an equilibrium.
Since~$\varphi^*$ is linear, it preserves translation equivalence.
\end{proof}

\begin{corollary}
Suppose that~$G$ is a \emph{generalized covering} of~$H$.
If the set~$\argequilibria{G}{f}$ is discrete (resp. finite),
then the set~$\argequilibria{H}{f}$ is discrete (resp. finite).
\end{corollary}

\begin{figure}
\centering
\includegraphics[align=c, width=0.36\columnwidth]{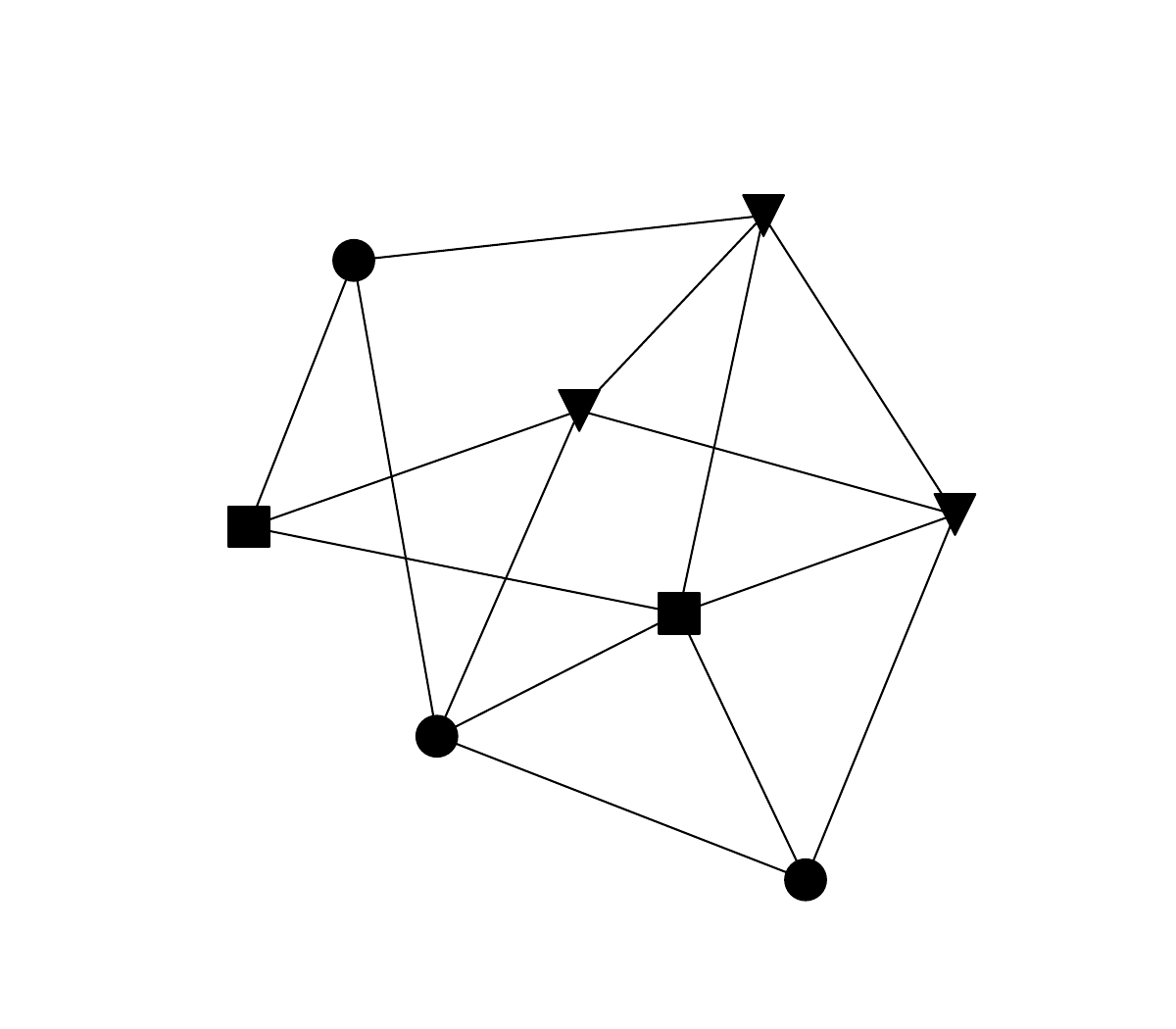}
$\to$
\includegraphics[align=c, width=0.25\columnwidth]{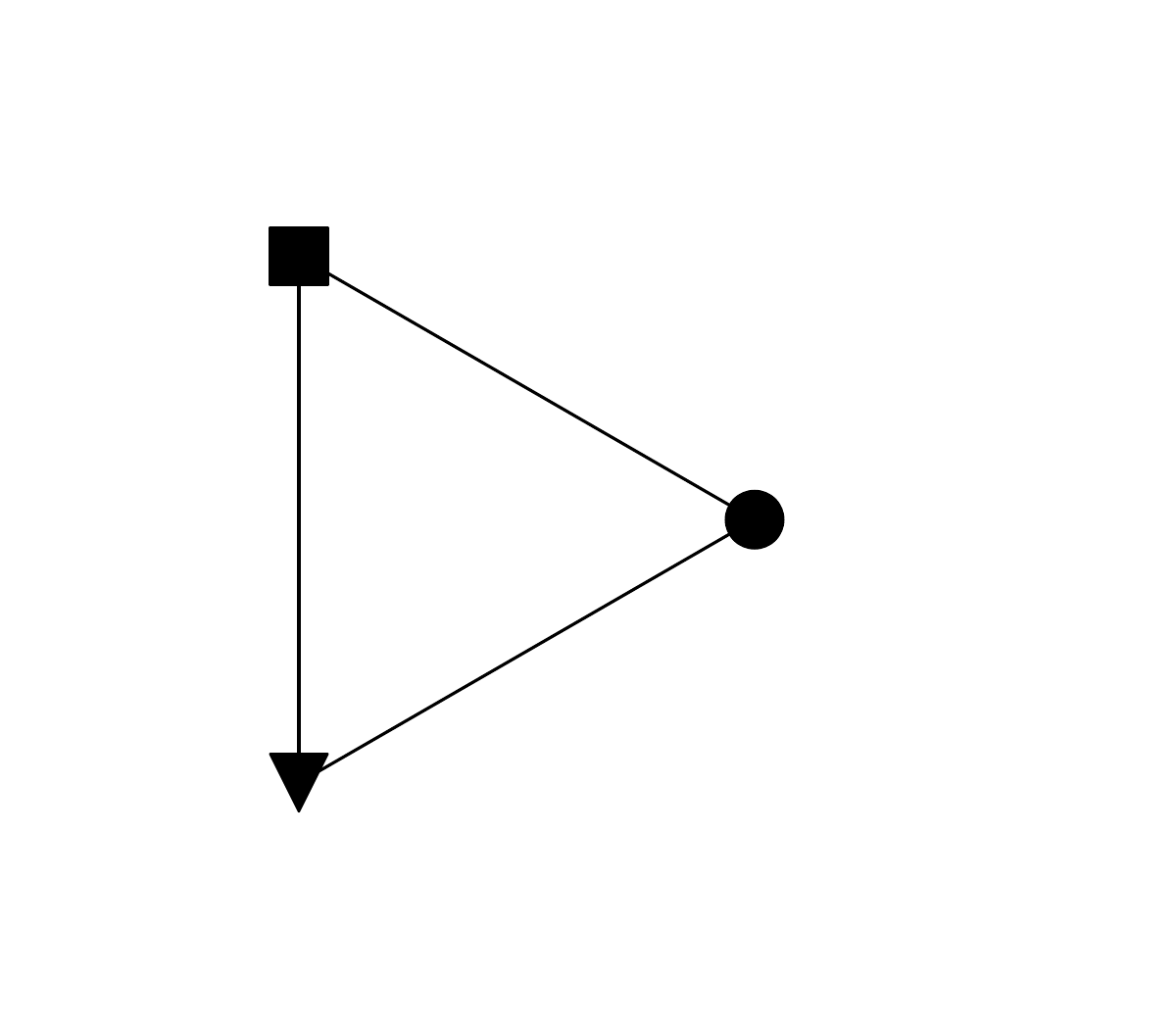}
\caption{
Generalized covering of the triangle graph.
Labels denote correspondences between vertices through the covering map.
}
\label{fig:asymmetric_covering}
\end{figure}

Theorem~\ref{thm:covering} allows to exhibit a rich family of graphs supporting
sets of equilibria of positive dimension.
Several examples of ordinary graph coverings of~$K_4$
can be found in~\cite{stark1996zeta, stark2000zeta, terras2007zeta}.
By Theorem~\ref{thm:covering} and Example~\ref{ex:sine_complete},
every covering of~$K_4$ supports a curve of equilibria for~$f(x)=\sin(x)$.

Let us give an example in the case of generalized coverings:

\begin{example} \label{ex:asymmetric_II}
Figure~\ref{fig:asymmetric_covering} shows that
the asymmetric graph of Figure~\ref{fig:book_asymmetric} in the introduction
is a generalized covering of~$K_3$.
Fix any~$a>0$ and let~$f(x)= ax-x^3$. For every~$\lambda$ small enough
the equation~$f(x)=\lambda$ has three distinct real solutions~$x_1,x_2,x_3$.
Since~$x_1+x_2+x_3=0$,
the configuration~$(0,x_1,x_1+x_2)$ is an equilibrium
of~$K_3$~\cite[Proposition 4.4]{homs2021nonlinear}.
As~$\lambda$ varies, this gives a curve of equilibria on~$K_3$.
By Theorem~\ref{thm:covering}, the curve of equilibria on~$K_3$
embeds to a curve of equilibria on the asymmetric graph.
\end{example}

%%%%%%%%%%%%%%%%%%%%%%%%%%%%%%%%%%%%%%%%%%%%%%%%%%%%%%%%%%%%%%%%%%%%%%%%%%%%%%%%%%%%%%%%%%%%%%%%
\subsection{Equilibria and Graph Automorphisms}
As shown by Example~\ref{ex:asymmetric_II}, asymmetric graphs can support
positive-dimensional sets of equilibria: graph symmetry is not the cause
of continuum of equilibria.
Rather, Example~\ref{ex:sine_complete}
suggests that graph symmetry can cause of singularities in the set of equilibria.

Let~$\aut(G)$ denote the automorphism group of the graph~$G$.
Recall that the elements~$\sigma \in \aut(G)$ are the bijections~$\sigma: V\to V$
preserving adjacency. The group~$\aut(G)$ acts naturally on the vector space~$\vertexspace$
by permuting coordinates as
\begin{equation} \label{eq:graph_symmetry}
	\sigma. (x_1,\ldots,x_m) \mapsto (x_{\sigma^{-1}(1)}, \ldots, x_{\sigma^{-1}(n)}).
\end{equation}
As well known in the case~$f(x)=\sin(x)$, see~\cite{ashwin1992dynamics},
this action preserves dynamics:

\begin{lemma}
The action~\eqref{eq:graph_symmetry} makes~$\aut(G)$ a group of symmetries of
the dynamical system~\eqref{eq:main}. In particular~$\aut(G)$ acts naturally on~$\equilibria$
as a group of diffeomorphisms.
\end{lemma}
\begin{proof}
Let~$\sigma \in \aut(G)$ and fix a vertex~$i\in V(G)$. The map~$\sigma$
induces a bijection between~$N(i)$ and~$N(\sigma(i))$.
Let~$\F$ denote the vector field of the system~\eqref{eq:main}.
Then
\[
	\F(\sigma.\x)_i
	= \sum_{j\in N(i)} f(x_{\sigma^{-1}(j)}-x_{\sigma^{-1}(i)})
	= \sum_{j\in N(\sigma^{-1}(i))} f(x_j-x_{\sigma^{-1}(i)})
	= \F(\x)_{\sigma^{-1}(i)},
\]
showing that the vector field~$\F$ is $\sigma$-equivariant.
\end{proof}

The discrete group of symmetries~$\aut(G)$ should not be confused with
the group of translational symmetries~$\homol_0(G)$ considered throughout the paper.
The joint action is given by~$\homol_0(G) \rtimes \aut(G)$.

The dynamics-preserving action of~$\aut(G)$ in vertex space~$\vertexspace$
induces a dynamics-preserving action of~$\aut(G)$ in vertex space~$\edgespace$.
In particular, the automorphism group~$\aut(G)$ acts on the set of
equilibria~$\equilibria$.

Recall that a point~$\x\in \equilibria$ is called regular if the set~$\equilibria$
is a manifold in a neighborhood of~$\x$, and it is called singular otherwise.
We have:

\begin{lemma} \label{lem:auto_singular}
Let~$\sigma\in \aut(G)$ and suppose that a point~$\x\in \equilibria$
is fixed by~$\sigma$. Then either~$\x$ is a singular point,
or~$\x$ is contained in a $\sigma$-invariant manifold of equilibria.
\end{lemma}
\begin{proof}
Let~$\sigma\in \aut(G)$. While~$\sigma$ acts on~$\vertexspace$
by coordinate permutations, it acts on~$\edgespace$ as a composition of a coordinate permutation
and possibly some change of signs,
due to the fact that the edge orientations in the definition of~$\edgespace$
might not be preserved by graph automorphisms. In both cases~$\sigma$
acts as an isometry, preserving the Euclidean distance.

Suppose that the fixed point~$\x$ is a regular point of~$\equilibria$.
Then there is an Euclidean open
ball~$U\subseteq \edgespace$ of~$\x$ such that~$U \cap \equilibria$
is a manifold. Since the action of~$\sigma$ on~$\edgespace$ is an isometry,
we have~$\sigma.U=U$. Since~$\sigma$ preserves~$\equilibria$,
we obtain~$\sigma.(U\cap \equilibria) = U\cap \equilibria$.
Therefore the set~$U\cap \equilibria$ is a $\sigma$-invariant manifold.
\end{proof}

%%%%%%%%%%%%%%%%%%%%%%%%%%%%%%%%%%%%%%%%%%%%%%%%%%%%%%%%%%%%%%%%%%%%%%%%%%%%%%%%%%%%%%%%%%%%%%%%

\subsection{Equilibria and Zeros of the Coupling Function}
For every graph~$G$ and for every function~$f \in \anal$ 
the set~$\equilibria$ contains at least one point, namely~$\0$.
We will see that in order for~$\equilibria$ to contain
more than one point, the function~$f$ must vanish at some point~$x\neq 0$
and, in order for~$\equilibria$ to have positive dimension,
the function~$f$ must change sign at some point~$x\neq 0$.

\begin{lemma} \label{lem:skew_norm}
For every equilibrium~$\y \in \equilibria$ the following identity holds:
\begin{equation} \label{eq:skew_norm}
	y_1 f(y_1) + \cdots + y_m f(y_m) = 0.
\end{equation}
\end{lemma}
\begin{proof}
Since~$\equilibria = \f^{-1}(\homol_1) \cap \homol_1^\perp$
the vectors~$\f(\y)$ and~$\y$ are orthogonal.
\end{proof}

\begin{lemma} \label{lem:exp_equilibria}
Let~$G$ be a connected graph and let~$f\in \anal$.
Suppose that there is~$x\neq 0$ such that~$f(x)=0$. Then the set~$\equilibria$
contains at least~$2^{n-1}$ points.
\end{lemma}
\begin{proof}
We define an equilibrium~$\x\in \vertexspace$ as follows:
fix a vertex~$i$ and let~$x_i=0$;
for every other vertex~$j\in V\setminus\{i\}$ let~$x_j\in \{0,x\}$.
Then for every edge~$pq \in E$ we have~$x_p-x_q \in \{0, +x, -x\}$. Therefore~$f(x_p-x_q)=0$
and~$\x$ is an equilibrium. Since~$x_i=0$ is fixed, different choices of~$x_j\in \{0,x\}$
give non-equivalent equilibria up to symmetry.
\end{proof}

\begin{theorem} \label{thm:f_zeros_equilibria}
Let~$G$ be a graph with at least one edge and let~$f\in \anal$.
Let~$\mathcal Z(f)$ denote the set of zeros of the function~$f$.
The following facts are true:
\begin{enumerate} [(i)]
\item ~$\equilibria = \{\0\}$ if and only if~$\mathcal Z(f) = \{0\}$;
\item if~$f(x)\geq 0$ for every~$x>0$, or if~$f(x)\leq 0$ for every~$x>0$,
then~$\equilibria$ is discrete and~$\dim \equilibria = 0$;
\item if~$f$ is increasing, then
every trajectory converges to the equilibrium~$\0$.
\end{enumerate}
\end{theorem}

\begin{proof}
Suppose that~$\mathcal Z(f) = \{0\}$. Without loss of generality, we can suppose~$f(x)>0$
for every~$x>0$ (replacing~$f$ by~$-f$ does not change the set of equilibria).
It follows that~$xf(x)>0$ for every~$x\neq 0$ and by~\eqref{eq:skew_norm}
we have~$\equilibria = \{\0\}$.
Conversely, suppose that~$f(x)=0$ for some~$x\neq 0$. Apply Lemma~\ref{lem:exp_equilibria}
to a connected component of~$G$ containing an edge.

If~$f(x)\geq 0$ for every~$x>0$ then~$xf(x)\geq 0$ for all~$x\neq 0$
and by~\eqref{eq:skew_norm} it follows that~$\equilibria \subseteq \mathcal Z(f)^m$.
In particular~$\equilibria$ is discrete and $0$-dimensional.
The case~$f(x)\leq 0$ for every~$x>0$ is analogous.

Suppose that~$f$ is increasing. Then~$\mathcal Z(f) = \{0\}$ and therefore~$\0$
is the only equilibrium up to symmetry. Moreover, it follows that
the energy function~\eqref{eq:energy} is convex, and thus
every trajectory converges to only local minimum.
\end{proof}

From Theorem~\ref{thm:f_zeros_equilibria} it follows that
for increasing coupling functions
the long term behavior of the system is the same as in the linear case~$f(x)=x$.

%%%%%%%%%%%%%%%%%%%%%%%%%%%%%%%%%%%%%%%%%%%%%%%%%%%%%%%%%%%%%%%%%%%%%%%%%%%%%%%%%%%%%%%%%%%%%%%%
%%%%%%%%%%%%%%%%%%%%%%%%%%%%%%%%%%%%%%%%%%%%%%%%%%%%%%%%%%%%%%%%%%%%%%%%%%%%%%%%%%%%%%%%%%%%%%%%
%%%%%%%%%%%%%%%%%%%%%%%%%%%%%%%%%%%%%%%%%%%%%%%%%%%%%%%%%%%%%%%%%%%%%%%%%%%%%%%%%%%%%%%%%%%%%%%%

\section{Stability} \label{sec:stability}

%%%%%%%%%%%%%%%%%%%%%%%%%%%%%%%%%%%%%%%%%%%%%%%%%%%%%%%%%%%%%%%%%%%%%%%%%%%%%%%%%%%%%%%%%%%%%%%%

\subsection{Gradient Flows}
As shown by Lemma~\ref{lem:energy}, the dynamical system~\eqref{eq:main}
can be written as~$\dot \x = -\nabla \E(\x)$ where~$\E$, called \emph{energy function},
is given by~\eqref{eq:energy}.
Trajectories always evolve in the direction in which the energy decreases maximally.
In particular, the only periodic trajectories are the equilibria.
Moreover, equilibria are exactly the critical points of the energy~$\E$.
Since energy is minimized over time, it is natural to expect a relation between
stable equilibria and local minimizers. However, this relation is subtle.
For a moment, let us review stability in a general gradient dynamical system.

\begin{definition} \label{def:stable}
An equilibrium~$\x$ is called~\emph{Lyapunov stable} if
for every neighborhood~$U$ of~$\x$ there is a forward-invariant
neighborhood~$V$ of~$\x$ contained in~$U$.
An equilibrium~$\x$ is called~\emph{asymptotically stable} if it is Lyapunov stable
and there is a neighborhood~$V$ of~$\x$ such that for every~$\y\in V$
the trajectory~$\flow_t(\y)$ converges to~$\x$ as~$t$ converges to~$+\infty$.
An equilibrium~$\x$ is called~\emph{linearly stable} if all eigenvalues
have strictly negative real part.
\end{definition}

The given definition of Lyapunov stability is equivalent to the following:
for every~$\varepsilon>0$ there is~$\delta\in(0,\varepsilon)$
such that~$\flow_t(B_\delta(\x)) \subseteq B_\varepsilon (\x)$ for every~$t\geq 0$.
A third equivalent definition appears in~\cite{anosov1988dynamical}:
as~$\y$ converges to~$\x$, the forward time
trajectory~$(\flow_t(\y))_{t\geq 0}$ converges uniformly in~$t$ to the stationary
trajectory~$(\x)_{t\geq 0}$.

\begin{definition}
Let~$h$ be a real-valued function.
We say that~$\x$ is a \emph{local minimizer} of~$h$
if there is a neighborhood~$U$ of~$\x$ such that~$h(\y)\geq h(\x)$ for every~$\y\in U$.
We say that~$\x$ is a \emph{strict local minimizer} of~$h$
if there is a neighborhood~$U$ of~$\x$ such that~$h(\y)> h(\x)$ for every~$\y\in U\setminus\{\x\}$.
We say that~$\x$ is an \emph{isolated local minimizer} of~$h$
if there is a neighborhood of~$\x$ in which~$\x$ is the only local minimizer of~$h$.
\end{definition}

As pointed out in~\cite{de2019asymptotically}, the relation between stability
and minimality in gradient systems is subtle and
standard references contain mistakes.
Figure~\ref{fig:gradient_optimization} summarizes the relations:
(a) is proven in~\cite{absil2006stable};
(b) is trivial;
(c) is a theorem by Lyapunov~\cite[Section 4.2]{anosov1988dynamical};
(d) is trivial;
proving (e) is an exercise in analysis;
(f) is the content of~\cite[Theorem 1]{de2019asymptotically};
(g) follows from Hadamard-Perron Theorem~\cite[Section 4.1]{anosov1988dynamical};
(h) follows from the fact that the set~$\{\y: h'(\y)=0\}$ is real-analytic and
thus locally-finite union of manifolds~\cite[Theorem 1.2.10]{trotman2020stratification}.

\begin{figure} [h]
\begin{tikzcd}
\text{local minimizer}  \arrow[dotted, Leftrightarrow, r, "(a)"] 							&  \text{Lyapunov stable} \\
\text{strict local minimizer} \arrow[Rightarrow, u, "(b)"] \arrow[Rightarrow, ur, "(c)"] 
	\arrow[dotted, Rightarrow, d, shift left=2, "(h)"] & 
\text{asymptotically stable} \arrow[Rightarrow, u, "(d)"] \\
\text{isolated local minimizer} \arrow[Rightarrow, u, shift left = 2, "(e)"]
	\arrow[Leftrightarrow, ur, "(f)"] &  \text{linearly stable} \arrow[Rightarrow, u, "(g)"]
\end{tikzcd}
\caption{
Solid arrows hold for smooth energy functions.
Dotted arrows hold for analytic energy functions
but fail in general for smooth energy functions.
}
\label{fig:gradient_optimization}
\end{figure}
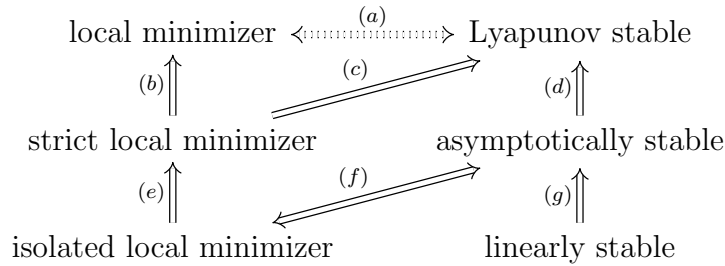

Let us now return to our system~\eqref{eq:main}.
While in Section~\ref{sec:set_of_equilibria} we worked mainly in edge space,
thus getting rid of translational symmetry,
we find discussing stability in vertex space more clear.
Recall that
the space~$\vertexspace$ is foliated in dynamically invariant subspaces of the form~$\x+\homol_0^\perp$,
all supporting the same dynamics, see Proposition~\ref{prop:sym}.
We say that a point~$\x$ is \emph{(local, strict, isolated) minimizer up to symmetry} if
it is (local, strict, isolated) minimizer of the restriction of~$\E$ to~$\x+\homol_0^\perp$.
Similarly, we say that an equilibrium~$\x$ is
\emph{(Lyapunov, asymptotically, linearly) stable up to symmetry}
if it has such property in the restricted dynamical system~$\x+\homol_0^\perp$.

\begin{proposition} \label{prop:stability}
Let~$G$ be a graph and let~$f\in \anal$.
The following facts are true:
\begin{enumerate} [(i)]
\item an equilibrium is Lyapunov stable if and only if it is Lyapunov stable up to symmetry
	and if and only if it is a local minimizer of~$\E$;
\item an equilibrium is asymptotically stable up to symmetry if and only if
	it is an isolated local minimizer of~$\E$ up to symmetry;
\item an equilibrium~$\x$ is linearly stable up to symmetry if and only if the matrix
\begin{equation} \label{eq:hessian}
	D^2 \E(\x) = B \mathrm{diag}(((f'(x_i-x_j))_{(i,j)\in E}) B^\trans
\end{equation}
is positive semidefinite and has rank $n-c$.
\end{enumerate}
\end{proposition}
\begin{proof}
Since the function~$f$ is analytic, the energy function~$\E$ is analytic, and thus
the first two statements follow from the general theory.
The matrix~\eqref{eq:hessian} is the hessian matrix~$D^2 \E(\x)$ of the energy function
and can be computed directly Corollary~\ref{cor:vector_form}.
Due to translational symmetry, for every equilibrium the eigenvalue~$0$ has multiplicity
at least equal to~$c$. The equilibrium~$\x$ is hyperbolic in
the invariant subspace~$\x + \homol_0^\perp$
if and only if the multiplicity of~$0$ is exactly~$c$.
\end{proof}

\begin{example} \label{ex:sine_complete_stability}
Let~$G$ be the complete graph with~$n\geq 4$ vertices and let~$f(x)=\sin(-x)$.
Compared to Example~\ref{ex:sine_complete}, we replaced~$\sin(x)$ by~$\sin(-x)$,
thus inverting time.
The energy~\eqref{eq:energy} can be written as
\[
	\E(\x) = \frac{1}{2} \sum_{j=1}^n\sum_{k=1}^n \cos(x_k - x_j) =
		\frac{1}{2} \abs{\sum_{j=1}^n e^{\mathrm{i} x_j}}^2.
\]
Therefore, the set of global minimizers is given by
\[
	\sum_{j=1}^n \cos(x_j) =
	\sum_{j=1}^n \sin(x_j) = 0.
\]
By Proposition~\ref{prop:stability} these points are Lyapunov stable. In particular,
the system contains an $(n-2)$-dimensional set of Lyapunov stable equilibria,
or $(n-3)$-dimensional set of Lyapunov stable equilibria up to symmetry.
Notice that these equilibria are not asymptotically stable up to symmetry, since
they are not isolated.
\end{example}

From Proposition~\ref{prop:stability} we obtain the following sufficient
condition for linear stability:

\begin{corollary} \label{cor:stability_positive_edges}
Let~$\x$ be an equilibrium and suppose that
for every edge~$(i,j)$ we have~$f'(x_i-x_j)>0$. Then the equilibrium is
linearly stable up to symmetry.
\end{corollary}
\begin{proof}
The matrix~\eqref{eq:hessian} is positive semidefinite with kernel equal to~$\ker B^\trans$,
and thus rank equal to~$n-c$. Proposition~\ref{prop:stability} applies.
\end{proof}

\begin{example}
Let~$G$ be any graph. Suppose that~$f'(0)>0$. By Corollary~\ref{cor:stability_positive_edges}
the equilibrium~$\x = \0$ is linearly stable up to symmetry.
\end{example}

A different method to analyze equilibrium stability, based on effective resistances
and inspired by the analysis of electrical networks,
is proposed in~\cite{homs2021nonlinear}.

%%%%%%%%%%%%%%%%%%%%%%%%%%%%%%%%%%%%%%%%%%%%%%%%%%%%%%%%%%%%%%%%%%%%%%%%%%%%%%%%%%%%%%%%%%%%%%%%

\subsection{Topological Bifurcation Theory}
Suppose that an equilibrium~$\x$ is contained in a $d$-dimensional manifold
of equilibria with~$d>c$. Then it is not isolated up to symmetry and, in general,
we cannot determine stability by eigenvalues. Indeed, in this case
the rank of the matrix~$D^2 \E(\x)$ is at most~$n-d$ and Proposition~\ref{prop:stability}
cannot be applied.

There are cases, however, in which stability can nevertheless determined by eigenvalues.
The idea is parametrizing the manifold of equilibria and analyzing how eigenvalues
changes along the manifold.
This is known as \emph{topological bifurcation theory},
or \emph{bifurcation without parameters}.
The following proposition allows to prove the existence of a manifold of stable equilibria
by computing the eigenvalues at a single point. However, it requires understanding
the geometry of the set of equilibria around the point.

\begin{proposition} \label{prop:stability_bifurcation}
Let~$\x \in \vertexspace$ be an equilibrium and~$d\geq 0$. Suppose that
\begin{enumerate} [(i)]
\item there is an open neighborhood~$U$ of~$\x$ such that~$U$ intersects
	the set of equilibria in a manifold of dimension~$d$;
\item the Hessian matrix~$D^2 \E(\x)$ has rank~$n-d$.
\end{enumerate}
Then the equilibrium~$\x$ is Lyapunov stable if and only
if the matrix~$D^2 \E(\x)$ is positive semidefinite.
Moreover, if the equilibrium~$\x$ is Lyapunov stable, then it is contained in a
$d$-dimensional manifold of Lyapunov stable equilibria.
\end{proposition}
\begin{proof}
Let~$\mathcal M$ denote set of equilibria contained in~$U$.
By hypothesis~$\mathcal M$ is a manifold.
The map~$\y \mapsto \mathrm{Rank} D^2 \E(\y)$ is lower-semicontinuous.
Therefore, up to restricting~$U$, we can suppose~$\mathrm{Rank} D^2 \E(\y) \geq n-2$
for every~$\y \in U$.
However, since~$\dim \mathcal M = d$, this implies that for every~$\y \in \mathcal M'$
we have~$\mathrm{Rank} D^2 \E(\y) = n-2$.
In other words, the manifold of equilibria~$\mathcal M$ is normally hyperbolic:
by Shoshitaishvili's Theorem
(see~\cite[Chapter 1]{liebscher2015bifurcation} or
the original paper~\cite{shoshitaishvili1972bifurcations})
the space is locally foliated by
$(n-d)$-dimensional
manifolds~$(\mathcal N_{\y})_{\y\in \mathcal M}$, such that each manifold~$\mathcal N_{\y}$
is orthogonal to~$\mathcal M$, with intersection~$\mathcal N_{\y} \cap \mathcal M = \{\y\}$,
and the equilibrium~$\y$ is hyperbolic in~$\mathcal N_{\y}$.
In follows that the equilibrium~$\x$ is Lyapunov stable if and only if
the matrix~$D^2 \E(\x)$ is positive semidefinite.
Moreover, since~$\mathrm{Rank} D^2 \E(\y) = n-2$ for every~$\y \in \mathcal M$,
it follows that if~$D^2 \E(\x)$ is positive semidefinite, then it remains positive semidefinite
in a neighborhood of~$\x$ in~$\mathcal M$.
\end{proof}

Notice that Proposition~\ref{prop:stability_bifurcation} applies to the regular points
of the set of equilibria, not to singular points. Where several
manifolds of equilibria intersect, the rank of~\eqref{eq:hessian} can drop.
This is confirmed by Example~\ref{ex:bifurcation}.

Proposition~\ref{prop:stability_bifurcation} concludes that, under certain hypothesis,
Lyapunov stability persists locally.
However, globally, a manifold of equilibria can exhibit change of stability, that is,
topological bifurcation.
Note that the gradient structure of the system restricts the possible
bifurcation phenomena by only allowing real eigenvalues.
For example, topological Hopf bifurcation,
considered in~\cite{fiedler2000generic1, fiedler2000generic2} cannot occur.

\begin{figure}
\centering
\includegraphics[width=0.6\columnwidth]{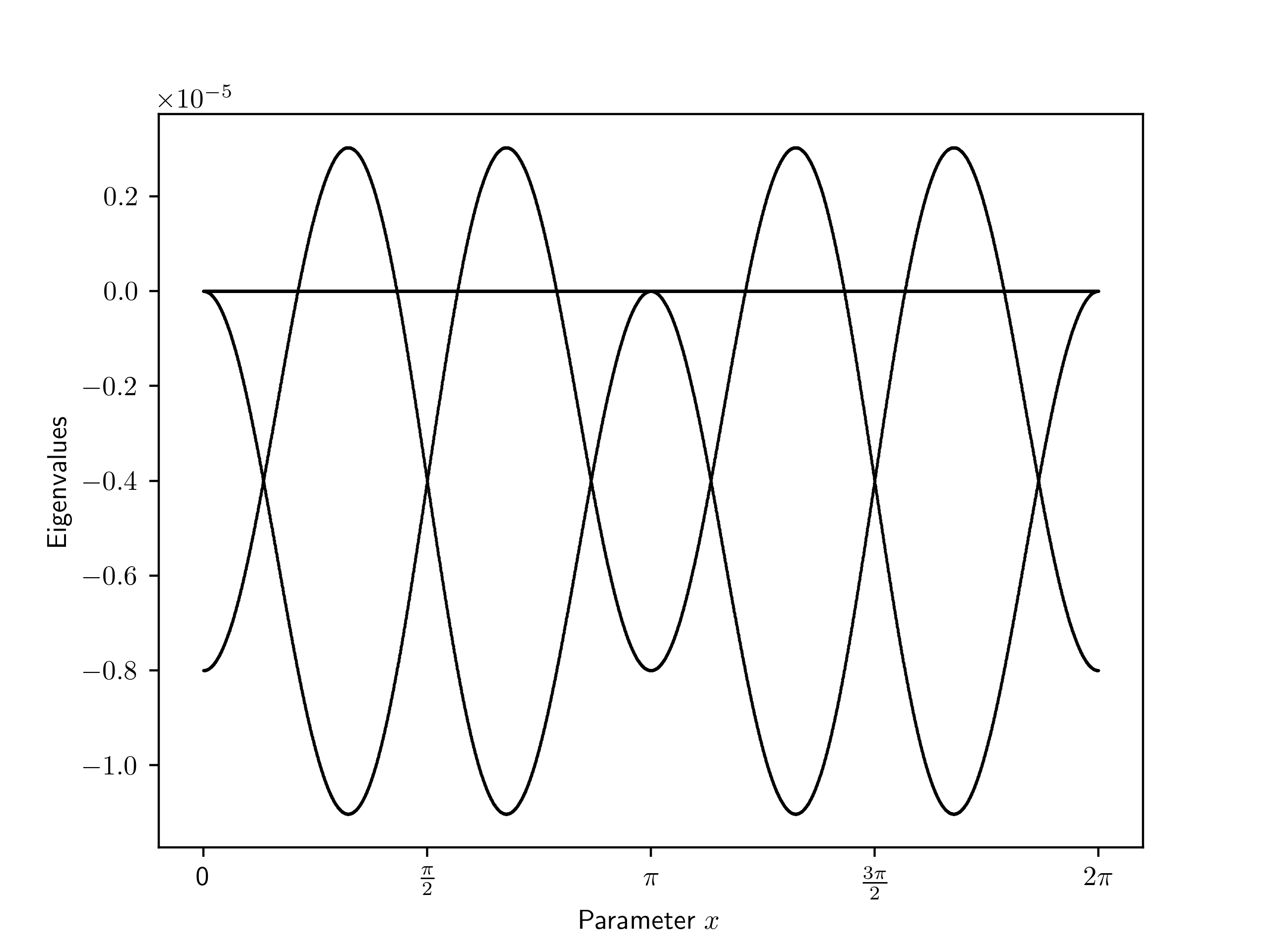}
\caption{
Let~$f(x)=\sin(x)-\sin(3x)$ and~$G=K_4$.
As~$x$ varies in~$[0,2\pi]$, the equilibrium~$(0,x,\pi,\pi+x)$ changes stability.
}
\label{fig:bifurcation}
\end{figure}

\begin{example} \label{ex:bifurcation}
Let~$f(x)=\sin(x)-\sin(3x)$ and~$G=K_4$.
As~$x$ varies in~$\R$ the configuration~$(0,x,\pi,\pi+x)$ traces a curve of equilibria.
As shown in Figure~\ref{fig:bifurcation}, equilibrium stability changes as~$x$ varies,
alternating between saddle and stable equilibrium.
For every~$x$ two eigenvalues are equal to~$0$. This is due to the curve of equilibria
and translational symmetry (with the notation of Proposition~\ref{prop:stability_bifurcation},
we have~$n=4$ and~$d=2$).
At~$x=\pi$ an extra eigenvalue vanishes without changing equilibrium stability.
This degeneracy is due to the symmetry group~$\aut(K_4)$:
at~$x=\pi$ the curve of equilibria~$(0,x,\pi,\pi+x)$
intersects another curve of equilibria, namely~$(0,\pi,x,\pi+x)$.
A similar phenomenon occurs at~$x=0$.
\end{example}

%%%%%%%%%%%%%%%%%%%%%%%%%%%%%%%%%%%%%%%%%%%%%%%%%%%%%%%%%%%%%%%%%%%%%%%%%%%%%%%%%%%%%%%%%%%%%%%%

\subsection{Block Decomposition Preserves Stability}
Proposition~\ref{prop:blocks} in Section~\ref{sec:preliminaries} shows that
the block decomposition of a graph~$G$ into its $2$-vertex connected components~$G_1,\ldots,G_p$
preserves equilibria, in the sense that~$\x \in \R^{V(G)}$ is an equilibrium on the graph~$G$
if and only if the restrictions~$\x\vert_{V(G_k)} \in \R^{V(G_k)}$ are equilibria
for each block~$G_k$.
The following result shows that block decomposition preserves
Lyapunov, asymptotic and linear stability.
Together with Proposition~\ref{prop:blocks}, this implies
Theorem~\ref{thm:intro_blocks_stability}
stated in the introduction.
Effectively, this result allows to reduce the analysis of equilibria and stability
to $2$-vertex-connected graphs.

\begin{theorem} \label{thm:blocks_stability}
Suppose that~$f\in \anal$ and that the graph~$G$ is connected.
Let
$
	G = G_1 \cup \cdots \cup G_p
$
be the block composition of~$G$.
An equilibrium~$\x \in \R^{V(G)}$ is Lyapunov stable in~$G$
if and only if for every block~$G_k$ the equilibrium~$\x\vert_{V(G_k)} \in \R^{V(G_k)}$
is Lyapunov stable in~$G_k$. 
The statement remains true if ``Lyapunov stable'' is replaced by
``asymptotically stable equilibrium up to symmetry'',
or ``linearly stable equilibrium up to symmetry''.
\end{theorem}

\begin{proof}
It is enough to prove the statement for~$p=2$. Suppose that~$G$ is the union of
two connected graphs~$H$ and~$K$ intersecting in a vertex~$k$. Up to relabeling the vertices
we have
\begin{align*}
	\x\vert_{H} = (x_1,x_2,\ldots,x_k), \quad
	\x\vert_{K} = (x_k,x_{k+1},\ldots,x_n).
\end{align*}
Proposition~\ref{prop:blocks} shows that~$\x$ is an equilibrium of~$G$ if and only
if~$\x\vert_{H}$ and~$\x\vert_{K}$ are equilibria of~$H$ and~$K$ respectively.
By definition of energy~\eqref{eq:energy} we obtain
\begin{equation} \label{eq:k_energy}
	\E_G(x_1,\ldots,x_n) =
		\E_{H} (x_1,\ldots,x_k) + \E_{K} (x_k,\ldots,x_n).
\end{equation}
Throughout the proof, we call \emph{perturbation} of a point~$\x$ a point~$\x'$
that can be chosen arbitrarily close to~$\x$.

Let us first prove the statement in the case Lyapunov stability.
From~\eqref{eq:k_energy} it follows
that if~$\x\vert_{H}$ and~$\x\vert_{K}$ are local minimizers
of~$\E_{H}$ and~$\E_{K}$ respectively,
then~$\x$ is a local minimizer of~$\E_G$.
Conversely, suppose that~$\x\vert_{H}$ is not a local minimizer of~$\E_H$.
Then there is perturbation~$\x'\vert_H = (x'_1,\ldots,x'_k)$
of~$\x\vert_H$ that makes~$\E_H$ decrease.
Up to translational symmetry in~$\R^{V(H)}$, we can suppose~$x'_k=x_k$.
Define
\begin{equation} \label{eq:k_lyapunov}
	\x' = (x'_1,\ldots,x'_{k-1},x_k,x_{k+1},\ldots,x_n).
\end{equation}
Since~$\x'_K = \x_K$, it follows from~\eqref{eq:k_energy}
that~$\x'$ is a perturbation of~$\x$ that makes the energy~$\E_G$ decrease.
It follows that~$\x$ is not a local minimizer of~$\E_G$.
An analogous argument shows that if~$\x\vert_{H}$ is not a local minimizer of~$\E_H$,
then~$\x$ is not a local minimizer of~$\E_G$.
This completes the proof for Lyapunov stable equilibria.

Let us prove the statement for asymptotic stability up to symmetry.
Suppose that~$\x\vert_{H}$ and~$\x\vert_{K}$ are strict local minimizers
up to symmetry in the respective systems. Take a perturbation~$\x'$ of~$\x$.
We want to prove that either~$\E_G(\x')>\E_G(\x)$ or~$\x'$ and~$\x$ are equivalent
up to translational symmetry in~$G$.
We have~$\E_H(\x'\vert_H)\geq \E_H(\x\vert_H)$ and~$\E_K(\x'\vert_K)\geq \E_K(\x\vert_K)$.
If any of the these two inequalities is strict, then~$\E_G(\x')>\E_G(\x)$.
Otherwise~$\E_H(\x'\vert_H)= \E_H(\x\vert_H)$ and~$\E_K(\x'\vert_K)= \E_K(\x\vert_K)$,
and since~$\x\vert_{H}$ and~$\x\vert_{K}$ are strict local minimizers up to symmetry,
it follows that there are~$\lambda, \mu\in \R$ such that
\[
	(x'_1,\ldots,x'_k) = (x_1 + \lambda, \ldots, x_k+\lambda), \quad
	(x'_k,\ldots,x'_n) = (x_k + \mu, \ldots, x_n+\mu).
\]
From these identities it follows that~$\lambda=\mu$. Therefore~$\x'$ and~$\x$
are equivalent up to translational symmetry in~$G$.
Conversely, suppose that~$\x_H$ is not a strict local minimizer of~$H$ up to symmetry.
Then there is perturbation~$\x'\vert_H = (x'_1,\ldots,x'_k)$ of~$\x\vert_H$
such that~$\E_H(\x') \leq \E_H(\x)$ with~$\x'\vert_H$ and~$\x\vert_H$
not equivalent up to translational symmetry in~$H$.
Up to translational we can suppose~$x'_k=x_k$. Now define~$\x'$
as in~\eqref{eq:k_lyapunov}, notice that~$\E_G(\x')\leq \E_G(\x)$
and that the perturbation~$\x'$ is not equivalent to~$\x$ up to translational symmetry in~$G$.

It remains to prove the statement for linearly stability up to symmetry.
Let~$\F$ denote the vector field and~$D\F(\x)$ linearization at~$\x$.
For every~$\v \in \vertexspace$
\begin{equation} \label{eq:k_laplacian}
	\v^T D\F_G (\x) \v =
    	\underbrace{\sum_{(i,j)\in E(H)} f'(x_i-x_j) (v_i-v_j)^2}_{
    		\v\vert_H^T D\F_H (\x\vert_H) \v\vert_H} +
    	\underbrace{\sum_{(i,j)\in E(K)} f'(x_i-x_j) (v_i-v_j)^2}_{
    		\v\vert_K^T D\F_K (\x\vert_K) \v\vert_K}.
\end{equation}
Suppose that~$\x_H$ and~$\x_K$ are linearly stable up to symmetry in~$H$ and~$K$ respectively.
Then~$D\F_H (\x\vert_H)$ and~$D\F_K (\x\vert_K)$ are negative semidefinite,
and therefore~$D\F_G (\x)$ is negative semidefinite. Moreover,
if~$\v^T D\F_G (\x) \v = 0$, then~$\v\vert_H^T D\F_H (\x\vert_H) \v\vert_H = 0$
and~$\v\vert_K^T D\F_K (\x\vert_K) \v\vert_K = 0$. It follows that there are~$\lambda,\mu \in \R$
such that~$\v\vert_H = (\lambda,\ldots,\lambda)$ and~$\v\vert_H = (\mu,\ldots,\mu)$.
Since the vertex~$k$ is in common, then~$\lambda=\mu$. This shows that the kernel of~$D\F_G (\x)$
has dimension~$1$. We conclude that~$\x$ is linearly stable up to symmetry.
Conversely, suppose that~$D\F_H (\x\vert_H)$ is not negative definite up to symmetry.
There are two possible reasons for this to occur.
First, the quadratic form~$\v\vert_H^T D\F_H (\x\vert_H) \v\vert_H$
has some positive value, or, second,
the kernel contains a vector whose coefficients are not all equal.
If~$\v\vert_H^T D\F_H (\x\vert_H) \v\vert_H > 0$
for some~$\v\vert_H\in \R^{V(H)}$,
then from~\eqref{eq:k_laplacian} it follows that the vector
\begin{equation} \label{eq:k_vector}
	\v = (v_1,\ldots,v_{k-1},v_k,v_k,\ldots,v_k)
\end{equation}
satisfies~$\v^T D\F_G (\x) \v>0$, showing that~$D\F_G (\x)$ is not negative semidefinite.
If
\[
	\v\vert_H^T D\F_H (\x\vert_H) \v\vert_H = 0
\]
for some vector~$\v\vert_H$ whose coefficients
are not all equal to each other, then the vector~$\v$ defined as in~\eqref{eq:k_vector}
satisfies~$\v^T D\F_G (\x) \v = 0$ and its coefficients are not all equal to each other.
This completes the proof.
\end{proof}

%%%%%%%%%%%%%%%%%%%%%%%%%%%%%%%%%%%%%%%%%%%%%%%%%%%%%%%%%%%%%%%%%%%%%%%%%%%%%%%%%%%%%%%%%%%%%%%%
%%%%%%%%%%%%%%%%%%%%%%%%%%%%%%%%%%%%%%%%%%%%%%%%%%%%%%%%%%%%%%%%%%%%%%%%%%%%%%%%%%%%%%%%%%%%%%%%
%%%%%%%%%%%%%%%%%%%%%%%%%%%%%%%%%%%%%%%%%%%%%%%%%%%%%%%%%%%%%%%%%%%%%%%%%%%%%%%%%%%%%%%%%%%%%%%%

\section{Open Problems} \label{sec:discussion}
If the function~$f$ is a polynomial, the set of equilibria is an algebraic set.
In this case, methods from algebraic geometry might lead to more precise results.

\begin{problem}
Classifying the graphs~$G$ and polynomials~$f$ such that~$\dim \equilibria>0$.
\end{problem}

However, a significant challenge arises due to the fact that the field~$\R$
is not algebraically closed.
For example, while the Gr\"obner basis allows to decide whether~$\dim X=0$
or~$\dim X >0$ for an algebraic set~$X$ in an algebraically closed field,
the same is not true in the field of real numbers.

We conjecture that for a generic graph~$G$ and a generic function~$f$
the set of equilibria up to symmetry~$\equilibria$ is discrete:

\begin{conjecture}
As~$n\to \infty$, for asymptotically almost every graph~$G$ with~$n$ vertices,
the set of non-zero real-analytic functions~$f$ such that~$\dim \equilibria>0$ is finite.
\end{conjecture}

Note that this conjecture is not in contrast with Example~\ref{ex:poly_cycle}
(resp. Example~\ref{ex:skew_book}) in which, for \emph{specific} graphs,
the inequality~$\dim \equilibria>0$ is satisfied by a $1$-dimensional
(resp. infinite-dimensional) family of~$f$.

Particular cases such as~$f(x)=\sin(x)$ and~$f(x)=x-x^3$ have been proposed
to model physical phenomena.
For example, the latter models polarization
in opinion dynamics~\cite{devriendt2021nonlinear}.
The existence of a continuum of equilibria in such models,
a structurally unstable phenomenon, can be interpreted
as a limitation of the model, or as valuable insight into the underlying physical phenomena.
In specific applications, it might be interesting to see what dynamics bifurcates from
such structures, for example by adding noise to the model.

\begin{problem}
Describing the dynamics near manifolds of equilibria
if additive Gaussian noise is added to~\eqref{eq:main}.
\end{problem}

%Theorem~\ref{thm:f_zeros_equilibria}
%shows that replacing the linear function~$f(x)=x$ by a nonlinear
%function like~$f(x)=\arctan(x)$ or~$f(x)=x^3$, does not alter
%the long-term dynamics of the system.
%This finding might have implications in image processing, where the graph heat equation is
%utilized for image smoothing~\cite{zhang2008graph}.
%Nonlinear functions offer greater flexibility and might help overcome issues
%related to over-smoothing, thereby preserving finer image details more effectively.

%%%%%%%%%%%%%%%%%%%%%%%%%%%%%%%%%%%%%%%%%%%%%%%%%%%%%%%%%%%%%%%%%%%%%%%%%%%%%%%%%%%%%%%%%%%%%%%%
%%%%%%%%%%%%%%%%%%%%%%%%%%%%%%%%%%%%%%%%%%%%%%%%%%%%%%%%%%%%%%%%%%%%%%%%%%%%%%%%%%%%%%%%%%%%%%%%
%%%%%%%%%%%%%%%%%%%%%%%%%%%%%%%%%%%%%%%%%%%%%%%%%%%%%%%%%%%%%%%%%%%%%%%%%%%%%%%%%%%%%%%%%%%%%%%%

\section*{Acknowledgements}
This paper was written while the author was a PhD candidate at the Vrije Universiteit Amsterdam.
The author is thankful to the institution for its support.
Special thanks go to Eddie Nijholt for valuable discussions, in particular
for pointing out implication~(h) in Figure~\ref{fig:gradient_optimization} and
for spotting several typos.

%%%%%%%%%%%%%%%%%%%%%%%%%%%%%%%%%%%%%%%%%%%%%%%%%%%%%%%%%%%%%%%%%%%%%%%%%%%%%%%%%%%%%%%%%%%%%%%%
%%%%%%%%%%%%%%%%%%%%%%%%%%%%%%%%%%%%%%%%%%%%%%%%%%%%%%%%%%%%%%%%%%%%%%%%%%%%%%%%%%%%%%%%%%%%%%%%
%%%%%%%%%%%%%%%%%%%%%%%%%%%%%%%%%%%%%%%%%%%%%%%%%%%%%%%%%%%%%%%%%%%%%%%%%%%%%%%%%%%%%%%%%%%%%%%%

\bibliographystyle{siam}
\bibliography{refs}

%%%%%%%%%%%%%%%%%%%%%%%%%%%%%%%%%%%%%%%%%%%%%%%%%%%%%%%%%%%%%%%%%%%%%%%%%%%%%%%%%%%%%%%%%%%%%%%%
%%%%%%%%%%%%%%%%%%%%%%%%%%%%%%%%%%%%%%%%%%%%%%%%%%%%%%%%%%%%%%%%%%%%%%%%%%%%%%%%%%%%%%%%%%%%%%%%
%%%%%%%%%%%%%%%%%%%%%%%%%%%%%%%%%%%%%%%%%%%%%%%%%%%%%%%%%%%%%%%%%%%%%%%%%%%%%%%%%%%%%%%%%%%%%%%%

\end{document}